\documentclass[a4paper, 11pt]{amsart}

\usepackage{amsmath,amssymb,amsthm,url,scalerel}
\usepackage[utf8]{inputenc}
\usepackage[T1]{fontenc}
\usepackage{libertine}
\usepackage[libertine,cmintegrals,cmbraces]{newtxmath}
\usepackage{verbatim}
\usepackage[shortlabels]{enumitem}
\usepackage{stmaryrd}
\usepackage{mathtools}
\usepackage{microtype}
\usepackage{tabto}
\usepackage{xcolor}
\usepackage{tikz}
\usepackage{tikz-cd}
\usepackage{nicefrac}
\usepackage{tabularx}
\usepackage{dsfont}
\usepackage{todonotes}
\usepackage{quiver}
\usepackage{a4wide}

\AtBeginDocument{%
   \def\MR#1{}
} 
\usepackage{euscript}
\usepackage[all]{xy}

\usepackage[pagebackref]{hyperref}
  
\hypersetup{%
  bookmarksnumbered=true,
  colorlinks=true,%
  linkcolor=blue,%
  citecolor=blue,%
  filecolor=blue,%
  menucolor=blue,%
  urlcolor=blue,%
  pdfnewwindow=true,%
  pdfstartview=FitBH}

\newcommand*\circled[1]{\tikz[baseline=(char.base)]{
            \node[shape=circle,draw,inner sep=1pt] (char) {#1};}}

\usepackage{xcolor}
\definecolor{seagreen}{RGB}{46,139,87}
\definecolor{maroon}{RGB}{128,0,0}
\definecolor{darkviolet}{RGB}{148,0,211}
\definecolor{twelve}{RGB}{100,100,170}
\definecolor{thirteen}{RGB}{100,150,50}
\definecolor{fourteen}{RGB}{200,0,0}
\definecolor{fifteen}{RGB}{0,200,0}
\definecolor{sixteen}{RGB}{0,0,200}
\definecolor{seventeen}{RGB}{200,0,200}
\definecolor{eighteen}{RGB}{0,200,200}

\newcommand{\bb}[1]{\mathbb{#1}}

\newcommand{\mbf}[1]{\mathbf{#1}}

\newcommand{\es}[1]{\EuScript{#1}}
\renewcommand{\sf}[1]{\mathsf{#1}}

\DeclareMathOperator{\colim}{\mathrm{colim}}

\DeclareMathOperator{\hocolim}{\mathrm{hocolim}}
\DeclareMathOperator{\holim}{\mathrm{holim}}
\DeclareMathOperator{\hofibre}{\mathrm{hofib}}

\newcommand{\s}{{\sf{Sp}}}
\DeclareMathOperator{\T}{\sf{Top}_\ast}

\newcommand{\poly}[1]{\mathsf{Poly}^{\leq #1}}
\newcommand{\homog}[1]{\mathsf{Homog}^{#1}}

\newcommand{\Aut}{\mathrm{Aut}}
\newcommand{\Fun}{\sf{Fun}}

\DeclareMathOperator{\Map}{\mathsf{Map}}

\newcommand{\cofrep}{\sf{Q}}
\newcommand{\fibrep}{\sf{R}}

\DeclareMathOperator{\id}{\mathrm{Id}}
\DeclareMathOperator{\ind}{\mathsf{ind}}
\DeclareMathOperator{\res}{\mathsf{res}}

\newcommand{\R}{\mathbb{R}}
\newcommand{\C}{\mathbb{C}}

  \newcommand{\adjunction}[4]{
\xymatrix{
#1:#2 \ar@<.5ex>[r] &
\ar@<.5ex>[l] #3:#4
}}

\newtheorem{thm}{Theorem}[subsection]
\newtheorem{prop}[thm]{Proposition}
\newtheorem{lem}[thm]{Lemma}
\newtheorem{cor}[thm]{Corollary}

\newtheorem*{thm*}{Theorem}

\theoremstyle{definition}
\newtheorem{definition}[thm]{Definition}

\newtheorem{rem}[thm]{Remark}

\begin{document}


\title{Comparing orthogonal calculus and calculus with Reality}

\author{Niall Taggart}
\address{Mathematical Institute, Utrecht University, Budapestlaan 6, 3584 CD Utrecht, The Netherlands}
\email{n.c.taggart@uu.nl}

\date{\today}



\begin{abstract}
We show that there exists a suitable $C_2$-fixed points functor from calculus with Reality to the orthogonal calculus of Weiss which recovers orthogonal calculus ``up to a shift'' in an analogous way with the recovery of real topological $K$-theory from Atiyah's $K$-theory with Reality via  appropriate $C_2$-fixed points.
\end{abstract}
\maketitle

\setcounter{tocdepth}{1}
{\hypersetup{linkcolor=black} \tableofcontents}

\section{Introduction}

\subsection*{Background} Orthogonal calculus~\cite{We95} is a homotopy theoretic tool developed to study geometric topology. Weiss' key observation, stemming from work of Weiss and Williams~\cite{WW88}, is that many objects of interest to geometric topologists may be packaged as functors from the category of Euclidean spaces and linear isometries to the category of (based) spaces. Key examples include the functor which assigns to a Euclidean space $V$ the classifying space $\sf{BDiff}(V)$ of the group of diffeomorphisms of $V$, or the functor which assigns to a Euclidean space $V$ the space of (smooth) embeddings $\sf{Emb}(M \times V, N \times V)$ for $M$ and $N$ fixed smooth manifolds, possibly with boundary conditions.  

To a functor $F$ from the category of Euclidean spaces to (based) spaces, the calculus assigns a tower of functors
\[\begin{tikzcd}
	&&& F \\
	\cdots & {T_{n}F} & {T_{n-1}F} & \cdots & {T_1F} & {T_0F}
	\arrow[from=2-2, to=2-3]
	\arrow[from=2-3, to=2-4]
	\arrow[from=2-1, to=2-2]
	\arrow[from=2-4, to=2-5]
	\arrow[from=2-5, to=2-6]
	\arrow[curve={height=-6pt}, from=1-4, to=2-6]
	\arrow[curve={height=-6pt}, from=1-4, to=2-5]
	\arrow[curve={height=6pt}, from=1-4, to=2-3]
	\arrow[curve={height=6pt}, from=1-4, to=2-2]
\end{tikzcd}\]
under $F$, which we refer to as \emph{the Weiss tower}. The $n$-th layer of the Weiss tower is the homotopy fibre $D_nF$ of the map $T_nF \to T_{n-1}F$ induced by the inclusion $\R^n \hookrightarrow \R^{n+1}$. Weiss~\cite{We95} classified the homotopy type of the $n$-th layers of this tower in terms of spectra with an action of the $n$-th orthogonal group $O(n)$, and the orthogonal calculus program aims to understand the homotopy type of the input functor $F$ in terms of the homotopy type of the spectra classifying the layers. Informally, one may think of this as a calculus in which total Stiefel-Whitney classes appear as first-order approximations and total Pontryagin classes as second-order approximations.

Orthogonal calculus fits into a larger framework of functor calculus, a categorification of Taylor's theorem from differential topology, pioneered by Goodwillie~\cite{Go90, Go91, Go03}. The aim of any version of functor calculus is to split a functor into ``polynomial'' parts, and reconstruct the homotopy type of the functor from these polynomial pieces. Depending on the natural structure of the functor you wish to study, one may want a calculus which captures this structure. For example a version of orthogonal calculus based on complex inner product spaces was of fundamental importance to Arone's study of Mitchell's finite spectra with $\mathcal{A}_k$-free cohomology~\cite{Ar98} which arise naturally as stabilisations of Weiss' \emph{cross-effects} or \emph{derivatives}, and to Hahn and Yuan's~\cite{HY19} study of multiplicative structures in the stable splittings of $\Omega \sf{SL}_n(\C)$. 

Since the original version of functor calculus there has been a small industry focused on the development of other versions of calculus that capture different structures, and in investigating the relationships between these various versions of functor calculus. Examples include the discrete calculus of Bauer, Johnson and McCarthy~\cite{BMJ} together with its comparison to Goodwillie calculus, and the manifold calculus of Goodwillie and Weiss~\cite{We96, We99, GW99} and its relationship with orthogonal calculus as exploited by Arone, Lambrechts and Voli\'c~\cite{ALV2} in their study of the stable rational homotopy type of spaces of embeddings. 

Motivated by the analogy between orthogonal calculus and real topological $K$-theory, the author developed a series of other calculi: unitary calculus~\cite{TaggartUnitary}, which sits in analogy with complex topological $K$-theory, and calculus with Reality~\cite{TaggartReality}, which is the calculus version of Atiyah's $K$-theory with Reality. These calculi are related in precisely the same manner as their analogous $K$-theories. The complexification-realification adjunction on the level of vector spaces induces a relationship between orthogonal and unitary calculus~\cite{TaggartOCandUC} which is well-behaved if the functors have convergent Weiss towers. This is the calculus version of the existence of maps (of spectra) 
\begin{align*}
&r: KU \to KO \\ &c: KO \to KU,
\end{align*}
relating real and complex topological $K$-theory. Moreover, we have a calculus version of the statement that the underlying non-equivariant spectrum of Atiyah's $K$-theory with Reality spectrum $KR$ is the complex topological $K$-theory spectrum $KU$; forgetting the ``complex conjugation'' action on calculus with Reality recovers unitary calculus~\cite{TaggartRealityUnitary}. 

\subsection*{Statement of main results} Figure \ref{figure 2} depicts the known relations between the various calculi.
\begin{figure}[ht]
{
\makebox[\textwidth][c]{
\xymatrix{
&
*+[F-:<3pt>]{\begin{array}{c}
\text{Calculus with} \\
\text{Reality} \end{array}} \ar[ddr]_{{\text{\cite{TaggartRealityUnitary}}}}^{\begin{array}{c}
\text{Forget} \\
C_2 - \text{action} \\
\end{array}} \ar@{-->}[ddl]_{\begin{array}{c}
C_2-\text{fixed} \\
\text{points} \\
\end{array}} & \\ & & & \\
*+[F-:<3pt>]{\begin{array}{c}
\text{Orthogonal calculus} \\ \end{array}} \ar@{<->}[rr]^{{\text{\cite{TaggartOCandUC}}}}_{\begin{array}{c}
\text{Complexification/} \\
\text{Realification} \\
\end{array}} & &
*+[F-:<3pt>]{\begin{array}{c}
\text{Unitary calculus} \\ \end{array}}
}
}
}
\caption{Relationship between various versions of Weiss calculus}
\label{figure 2}
\end{figure}
In this work we complete these comparisons by completing the dotted arrow in the diagram, i.e., by demonstrating the calculus version of the fact that the spectrum $KO$ of real topological $K$-theory is the (homotopy) $C_2$-fixed points of $KR$. We construct an appropriate fixed points functor from calculus with Reality to orthogonal calculus and show that this functor interacts well with the Weiss towers. 

The major technical difficulty to overcome is the non-existence of a good fixed points functor. Calculus with Reality is constructed with respect to the underlying non-equivariant equivalences of $C_2$-spaces, which are preserved by homotopy $C_2$-fixed points, but \emph{not} $C_2$-fixed points, in general. On the other hand, the polynomial approximations are constructed as a (sequential) homotopy colimit of a (compact) homotopy limit. These homotopy colimits are preserved by $C_2$-fixed points but \emph{not} by homotopy $C_2$-fixed points, in general. In some sense the problem is that our model for $C_2$-spaces is based on \emph{free} $C_2$-spaces, which interact poorly with the various versions of fixed points. We choose to model $C_2$-spaces by \emph{cofree} $C_2$-spaces, on which homotopy $C_2$-fixed points and $C_2$-fixed points agree, and hence we can exploit all of the nice properties of both functors. We then apply this to the equivalence between free and cofree $C_2$-spaces, see Proposition~\ref{prop: free=cofree}, to produce a cofree model for calculus with Reality.

Although we do not state it, our methods for showing that calculus with Reality based on cofree $C_2$-spaces agrees with calculus with Reality based on free $C_2$-spaces extend without much work to show that if $\es{C}$ is a topological model category on which there exists a suitable notion of Weiss calculus, and $\es{D}$ is a Quillen equivalent topological model category with left Quillen functor $L: \es{C} \to \es{D}$ which preserves all weak equivalences then there exists a suitable notion of Weiss calculus on $\es{D}$ which is equivalent to the Weiss calculus on $\es{C}$. This is similar to the analysis carried out by Walter~\cite{WalterRational} in Goodwillie calculus.

We exhibit an equivalence of homotopy theories between the polynomial functors in the free model for calculus with Reality constructed in~\cite{TaggartReality} and the homotopy theory of these functors based on cofree $C_2$-spaces. 

\begin{thm*}[Proposition~{\ref{prop: polynomial invariant under QE}}]
Let $n$ be a non-negative integer. The identity functor induces an equivalence of $\infty$-categories between the $\infty$-category of $n$-polynomial functors based on free $C_2$-spaces and the $\infty$-category of $n$-polynomial functors based on cofree $C_2$-spaces. 
\end{thm*}

We perform a similar analysis for homogeneous functors in the free model for calculus with Reality constructed in~\cite{TaggartReality} and the homotopy theory of these functors based on cofree $C_2$-spaces. 

\begin{thm*}[Proposition~{\ref{prop: homogeneous invariant under QE}}] Let $n$ be a non-negative integer. The identity functor induces an equivalence of $\infty$-categories between the $\infty$-category of $n$-homogeneous functors based on free $C_2$-spaces and the $\infty$-category of $n$-homogeneous functors based on cofree $C_2$-spaces. 
\end{thm*}

We further show that calculus with Reality may be modelled on cofree $C_2$-spaces by demonstrating that the equivalence of $\infty$-categories between the $\infty$-category of $n$-homogeneous functors and the $\infty$-category of spectra with an action of $C_2\ltimes U(n)$ coming from the classification of $n$-homogeneous functors naturally lifts to a suitable equivalence between the cofree models. The following result is a combination of the results in $\S$\ref{subsection: cofree differentiation}.

\begin{thm*}
There is a commutative diagram 
\[\begin{tikzcd}
	{\homog{n}(\es{J}_0^\mbf{R},C_2\T^\sf{free})} & {\homog{n}(\es{J}_0^\mbf{R},C_2\T^\sf{cofree})} \\
	{\Fun_{C_2 \ltimes U(n)}(\es{J}_n^\mbf{R}, (C_2\ltimes U(n)\T^\sf{free})} & {\Fun_{C_2 \ltimes U(n)}(\es{J}_n^\mbf{R}, (C_2\ltimes U(n)\T^\sf{cofree})} \\
	{\s^\mbf{R}[U(n)]^\sf{free}} & {\s^\mbf{R}[U(n)]^\sf{cofree}} \\
	{\s^\mbf{O}[C_2 \ltimes U(n)]^\sf{free}} & {\s^\mbf{O}[C_2 \ltimes U(n)]^\sf{cofree}}
	\arrow["{(\alpha_n^\mbf{R})^\ast}"', shift right=1, from=3-1, to=2-1]
	\arrow["{(\alpha_n^\mbf{R})_!}"', shift right=1, from=2-1, to=3-1]
	\arrow["\psi", shift left=2, from=3-1, to=4-1]
	\arrow["{L_\psi}", from=4-1, to=3-1]
	\arrow["{\res_0^n/U(n)}", shift left=1, from=2-1, to=1-1]
	\arrow["{\ind_0^n\varepsilon^\ast}", shift left=1, from=1-1, to=2-1]
	\arrow["{\mathds{1}}", shift left=1, from=1-1, to=1-2]
	\arrow["{\mathds{1}}", shift left=1, from=1-2, to=1-1]
	\arrow["{\mathds{1}}", shift left=1, from=2-1, to=2-2]
	\arrow["{\mathds{1}}", shift left=1, from=2-2, to=2-1]
	\arrow["{\mathds{1}}", shift left=1, from=3-1, to=3-2]
	\arrow["{\mathds{1}}", shift left=1, from=3-2, to=3-1]
	\arrow["{\mathds{1}}", shift left=1, from=4-1, to=4-2]
	\arrow["{\mathds{1}}", shift left=1, from=4-2, to=4-1]
	\arrow["{L_{\psi}}", shift left=1, from=4-2, to=3-2]
	\arrow["\psi", shift left=1, from=3-2, to=4-2]
	\arrow["{(\alpha_n^\mbf{R})_!}"', shift right=1, from=2-2, to=3-2]
	\arrow["{(\alpha_n^\mbf{R})^\ast}"', shift right=1, from=3-2, to=2-2]
	\arrow["{ind_0^n\varepsilon^\ast}", shift left=1, from=1-2, to=2-2]
	\arrow["{\res_0^n/U(n)}", shift left=1, from=2-2, to=1-2]
\end{tikzcd}\]
of adjoint pairs in which every adjoint pair induces an equivalence of $\infty$-categories. 
\end{thm*}

In fact, the preceding theorems are stronger than indicated: we construct particular model structures and provide explicit Quillen equivalences rather than an abstract equivalence of $\infty$-categories. Using the cofree model and exploiting the equivalence between homotopy fixed points and fixed points, we can directly compare the polynomial approximations. To distinguish between constructions made in orthogonal calculus and constructions made in calculus with Reality we invoke a superscript $\mbf{O}$ and $\mbf{R}$ to denote orthogonal calculus and calculus with Reality constructions, respectively. 

\begin{thm*}[Theorem {\ref{thm: polynomial approximation preserved}}]
Let $n$ be a non-negative integer. If $F$ is an fibrant functor in the cofree model for calculus with Reality, then there is a levelwise weak equivalence of orthogonal functors
\[
c^\ast(T_n^\mbf{R}F)^{C_2} \longrightarrow T_n^\mbf{O}(c^\ast F^{C_2}).
\]
\end{thm*}

Our comparison is complete by examining the layers of the respective Weiss towers. 

\begin{thm*}[Corollary {\ref{cor: layers agree}}]
Let $n$ be a non-negative integer. If $F$ is an fibrant functor in the cofree model for calculus with Reality, then there is a levelwise weak equivalence of orthogonal functors
\[
c^\ast(D_n^\mbf{R}F)^{C_2} \longrightarrow D_n^\mbf{O}(c^\ast F^{C_2}).
\]
\end{thm*}

To complete our study of the relationship between calculus with Reality and orthogonal calculus we provide a complete relationship between the higher categorical classifications for both calculi. This results in a diagram of $\infty$-categories (induced by a diagram of model categories), see Figure~\ref{fig: model categories}, which commutes up to natural equivalence, see Section~\ref{section: model categories}. In order to construct such a commutative diagram we construct ``multiplicity $n$'' versions of Schwede's comparison functors between orthogonal spectra and real spectra, see e.g.,~\cite[Example 7.11]{Sch19}.

\subsection*{Conventions}
We abuse language and refer to based compactly generated weak Hausdorff spaces as spaces, and similarly for based spaces with an action of a compact Lie group $G$, which we will refer to as $G$-spaces. 

We model $\infty$-categories by quasi-categories as introduced by Boardman and Vogt~\cite{BVQuasi} and developed extensively by Joyal~\cite{JoyalQuasi} and Lurie~\cite{LurieHTT}. Given a topological model category $\es{M}$, the category $\es{M}^\sf{fc}$ of fibrant-cofibrant objects defines a topological category, the nerve of which defines a quasi-category. We refer to the quasi-category $N(\es{M}^\sf{fc})$ as the \emph{underlying} $\infty$-category of $\es{M}$, and denote it by $\es{M}_\infty$. Given a Quillen adjunction (\emph{resp.} Quillen equivalence) of model categories $\adjunction{F}{\es{M}}{\es{N}}{G}$, there is a well-defined adjunction (\emph{resp.} equivalence) of underlying $\infty$-categories, $\adjunction{\mathds{L}F}{\es{M}_\infty}{\es{N}_\infty}{\mathds{R}G}$.

\subsection*{Acknowledgements}
The author wishes to thank L.~Pol for numerous helpful conversations about this work and the referee for numerous helpful suggestions which greatly improved this paper. A portion of this work was completed while the author was in residence at the Institut Mittag-Leffler in Djursholm, Sweden in 2022 as part of the program `Higher algebraic structures in algebra, topology and geometry' supported by the Swedish Research Council under grant no. 2016-06596.  The author was supported by the European Research council (ERC) through the grant “Chromatic homotopy theory of spaces”, grant no. 950048.

\section{A cofree model for calculus with Reality}

\subsection{Models for equivariant spaces}\label{section: equivariant spaces}

In the construction of calculus with Reality~\cite{TaggartReality}, we considered the \emph{underlying} or \emph{coarse} model structure on  based $C_2$-spaces in which a map $f: X \to Y$ is a weak equivalence or fibration if and only if $f$ is a weak equivalence or fibration in the Quillen model structure on based spaces. The cofibrant objects are the \emph{free} $C_2$-spaces, i.e., those spaces which are (up to underlying weak equivalence) of the form $(EC_2)_+ \wedge X$ for $X$ a based space. This model structure is the transfer of the Quillen model structure on based spaces along the adjoint pair
\[
\adjunction{(C_2)_+ \wedge (-)}{\T}{C_2\T}{i^\ast}
\]
and we denote this model structure by $C_2\T^\sf{free}$. In this section we present an alternative model for this homotopy theory which is presented by \emph{cofree} $C_2$-spaces, i.e., we give a model structure on $C_2\T$ in which the weak equivalences are still underlying equivalences but the fibrant objects are $C_2$-spaces of the form $\Map_\ast((EC_2)_+, X)$ for $X$ a based space. Much of what we say in this section is true for arbitrary groups $G$, but we specialise to $G=C_2$ as it is the only case of interest for us.

We start by recalling the \emph{fine} (or ``genuine'') model structure on $C_2$-spaces of which our cofree model structure will be a localization.

\begin{lem}[{\cite[Theorem III.1.8]{MM02}}]
There is a model structure on the category of $C_2$-spaces such that a map $f: X \to Y$ is a weak equivalence or fibration if and only if the induced map on $H$-fixed points $f^H : X^H \to Y^H$ is a weak equivalence or fibration of based spaces for all subgroups $H$ of $C_2$. We denote this model structure by $C_2\T^{\sf{fine}}$. This model structure is cellular, proper and topological. 
\end{lem}

We now state the \emph{cofree} model structure as a localization of the fine model structure, see also~\cite[Example 7.9]{LawsonBousfield}. In particular homotopy $C_2$-fixed points are precisely the right derived functor of $C_2$-fixed points in this model structure. 

\begin{prop}
There is a model structure on the category of $C_2$-spaces such that a map $f: X \to Y$ is a weak equivalence if it is an underlying weak homotopy equivalence or a cofibration if it is a cofibration in the fine model structure. The fibrant objects are the cofree $C_2$-spaces. We denote this model structure by $C_2\T^\sf{cofree}$. This model structure is cellular, proper and topological.
\end{prop}
\begin{proof}
	This model structure is the left Bousfield localization of the fine model structure at the map $(EC_2)_+ \to \ast$, which exists by~\cite[Theorem 4.1.1]{Hi03}. A space $X$ is $(EC_2)_+$-local if and only if the induced map
	\[
	X^{C_2} \longrightarrow X^{hC_2}
	\]
	is an underlying weak equivalence, i.e., if and only if $X$ is a cofree $C_2$-space. In particular, a localization functor is given by 
	\[
	X \longmapsto \Map_\ast((EC_2)_+, X).
	\]
	It follows that map $f: X \to Y$ is a weak equivalence if and only if 
	\[
	\Map_\ast((EC_2)_+, X)^H \longrightarrow \Map_\ast((EC_2)_+, Y)^H
	\]
	is a weak equivalence of based spaces for all $H \leq C_2$. This last is equivalent to $f: X \to Y$ and $f^{hC_2}:~X^{hC_2} \to Y^{hC_2}$ being underlying weak equivalences, which since homotopy fixed points preserve underlying weak equivalences, is equivalent to $f: X \to Y$ being an underlying weak equivalence. 
	
The general theory of left Bousfield localization implies all of the properties of this model structure except for right properness, which follows since right properness is determined by the weak equivalence class, and the free model structure on $C_2$-spaces is right proper with the same weak equivalences as in the cofree model structure, see e.g.,~\cite[Remark 3.5.6]{BalchinMCat}.
\end{proof}

\begin{prop}\label{prop: free=cofree}
The identity functor
\[
\mathds{1} : C_2\T^{\sf{free}} \longrightarrow C_2\T^{\sf{cofree}}
\]
is a left Quillen equivalence.
\end{prop}
\begin{proof}
The identity functor 
\[
\mathds{1}: C_2\T^\sf{free} \to C_2\T^\sf{fine}
\]
is a left Quillen functor from the free model structure to the fine model structure, and since the cofree model structure is a left Bousfield localization of the fine model structure, we obtain a left Quillen functor
\[
\mathds{1} : C_2\T^\sf{free} \longrightarrow C_2\T^\sf{cofree}.
\]
The weak equivalences in the free and cofree model structure agree, hence it suffices to show that either of the (co)units is a weak equivalence. This follows readily from the fact that for any $C_2$-space $X$, the map 
\[
(EC_2)_+ \wedge \Map_\ast((EC_2)_+, X) \longrightarrow X
\]
is an underlying weak equivalence. 
\end{proof}

\subsection{A cofree model for input functors}
Let $\mathds{k}$ denote either $\R$ or  $\C$ , and let $\es{U}$ denote the infinite-dimensional $\mathds{k}$ inner product spaces $\R^\infty$, $\C^\infty$ or $\C \otimes \R^\infty$. The infinite-dimensional inner product space $\es{U}$ is the indexing universe for the various versions of orthogonal calculus, e.g., $\es{U} = \R^\infty$ indexes orthogonal calculus and $\es{U} = \C \otimes \R^\infty$ indexes calculus with Reality. Define $\es{J}$ to be the category of finite-dimensional inner product subspaces of $\es{U}$ with morphism the linear isometries. The category $\es{J}$ is topologically enriched as $\es{J}(V,W)$ is the Stiefel manifold of $\dim(V)$-frames in $W$. By adding a disjoint basepoint to the morphism space produces a $\T$-enriched category, which we continue to denote by $\es{J}$. In the case of calculus with Reality this category is $C_2\T$-enriched with $C_2$ acting by complex conjugation, see~\cite[\S1.2]{TaggartReality}.

The input category for calculus with Reality is the category of $C_2$-equivariant continuous functors $\Fun(\es{J}^\mbf{R}, C_2\T)$ with $C_2$ acting by conjugation, while the input category for orthogonal calculus is the category of continuous functors $\Fun(\es{J}^\mbf{O}, \T)$. When discussing both calculi at once the phrase ``continuous functors'' should be interpreted in the appropriate sense. For calculus with Reality, the adjunction between free and cofree $C_2$-spaces induces an Quillen equivalence on the level of functor categories.

\begin{prop}\label{prop: input invariant under QE}
Let $n$ be a non-negative integer. The adjoint pair
\[
\adjunction{\mathds{1}}{\Fun(\es{J}^\mbf{R}, C_2\T^\sf{free})}{\Fun(\es{J}^\mbf{R}, C_2\T^\sf{cofree})}{\mathds{1}}
\]
is a Quillen equivalence.
\end{prop}
\begin{proof}
Since (acyclic) fibrations are defined levelwise in 	$\Fun(\es{J}, C_2\T^\sf{cofree})$, and the identity functor
\[
\mathds{1} : C_2\T^\sf{cofree} \longrightarrow C_2\T^\sf{free}
\]
is right Quillen, it follows that the induced adjunction on projective model structures is a Quillen adjunction.

To show that the adjunction is a Quillen equivalence it suffices to show that the derived (co)units are equivalences. Unravelling the definition of the (co)unit, one observes that the Quillen equivalence between $C_2\T^\sf{free}$ and $C_2\T^\sf{cofree}$ forces the (co)units to be equivalences on the level of functor categories.
\end{proof}

\subsection{A cofree model for polynomial functors}

For the remainder of this section we let $\es{C}$ denote either $\T$ or $C_2\T$. 

\begin{definition}\label{def: n-poly}
Let $n$ be a non-negative integer. A functor  $F: \es{J} \to \es{C}$ is said to be \emph{polynomial of degree less than or equal $n$} or equivalently \emph{$n$-polynomial} if for all $V \in \es{J}$, the canonical map
\[
F(V) \longrightarrow \underset{0 \neq U \subseteq \mathds{k}^{n+1}}{\holim}~F(V \oplus U)=:\tau_n F(V)
\]
is a weak equivalence in $\es{C}$.
\end{definition}

\begin{rem}
The poset of subspaces of $\mathds{k}^{n+1}$ is a topological poset in that it has a \emph{space} of object and \emph{space} of morphisms, with continuous source, target, composition and insertion of identity maps. The space of objects is a disjoint union of Grassmannian manifolds, while the space of morphisms is a disjoint union of flag manifolds.
\end{rem}

There is a model structure which captures the homotopy theory of $n$-polynomial functors, and hence in which fibrant replacement is $n$-polynomial approximation.

\begin{lem}[{\cite[Proposition 6.5 $\&$ Proposition 6.6]{BO13}},{\cite[Proposition 2.15]{TaggartReality}}]
Let $n$ be a non-negative integer. There is a cellular, proper and topological model category structure on the category of continuous functors $\Fun(\es{J}, \es{C})$ with cofibrations the projective cofibrations and fibrant objects are the levelwise fibrant $n$-polynomial functors. We denote this model structure by $\poly{n}(\es{J}, \es{C})$. 
\end{lem}

\begin{rem}
The $n$-polynomial model structure is the left Bousfield localization of the projective model structure at the set
\[
\left\{ \ \underset{0 \neq U \subset \mathds{k}^{n+1}}{\hocolim}~\es{J}(V \oplus U, -) \longrightarrow \es{J}(V,-) \ \mid \ V \in \es{J} \ \right\}.
\]
By the relevant version of \cite[Proposition 4.2]{We95} to the known version of Weiss calculus under consideration, there is an identification
\[
S\gamma_{n+1}(V,W)_+ \cong \underset{0 \neq U \subset \mathds{k}^n}{\hocolim}~\es{J}_0(V \oplus U, W)
\]
of the sphere bundle of $\gamma_{n+1}(V,W)$ with a topological homotopy colimit of Stiefel manifolds, hence the set of maps is equivalently described via the projection maps of the bundle $S\gamma_{n+1}(V,-)$. 
\end{rem}

\begin{rem}
The underlying $\infty$-category of the $n$-polynomial model structure is the $\infty$-category of $n$-polynomial functors. This should be equivalently defined as the full $\infty$-subcategory of the $\infty$-category of functors $\Fun(\es{J}, \es{C})$ spanned by the $n$-polynomial functors, if one suitably interprets Definition~\ref{def:  n-poly} in an appropriate $\infty$-categorical sense. 
\end{rem}

For calculus with Reality, the adjunction between free and cofree $C_2$-spaces induces a Quillen equivalence between $n$-polynomial model structures.

\begin{prop}\label{prop: polynomial invariant under QE}
Let $n$ be a non-negative integer. The adjoint pair
\[
\adjunction{\mathds{1}}{\poly{n}(\es{J}, C_2\T^\sf{free})}{\poly{n}(\es{J}, C_2\T^\sf{cofree})}{\mathds{1}}
\]
is a Quillen equivalence.
\end{prop}
\begin{proof}
Apply \cite[Theorem 3.3.20]{Hi03}, noticing that the localizing set of the cofree $n$-polynomial model structure is the derived image of the localizing set of the  free $n$-polynomial model structure, noting that the representable functors are necessarily cofibrant in any projective model structure.
\end{proof}

We now define the $n$-polynomial approximation functor, the idea being that iterations $(\tau_n)^k$ of $\tau_n$ become closer and closer to being $n$-polynomial as $k$ increases. 

\begin{definition}
Let $n$ be a non-negative integer. For any continuous functor $F: \es{J} \to \es{C}$, define the \emph{$n$-polynomial approximation} $T_nF$ of $F$ to be the sequential homotopy colimit
\[
\hocolim~(F \longrightarrow \tau_n(F) \longrightarrow (\tau_n)^2(F) \longrightarrow \cdots).
\]
\end{definition}

\begin{rem}
In all known versions of Weiss calculus the above homotopy colimit produces an $n$-polynomial functor, see, e.g.,~\cite{We98}. This is a sophisticated result which is the heart of much of the work in constructing new variants of Weiss calculus. For example, the only known proof that $T_nF$ is $n$-polynomial for orthogonal calculus requires: a calculation of the connectivity of the map $S\gamma_{n+1}(V,W) \to \es{J}(V,W)$; showing that the functor $\tau_n$ increases connectivity using an elaborate calculation of the connectivity of topological homotopy limits; and, that 
\[
T_n(S\gamma_{n+1}(V,W)) \to T_n(\es{J}(V,W)),
\]
is a weak equivalence, where $S\gamma_{n+1}(V,W)$ is the sphere bundle of the $(n+1)$-fold orthogonal complement bundle $\gamma_{n+1}(V,W)$ over $\es{J}(V,W)$ with fibre over $f$ given by $\mathds{k}^{n+1} \otimes f(V)^\perp$.  Weiss uses this fact to show that for every continuous functor $F: \es{J} \to \T$ and every non-negative integer $i$, the square
\[\begin{tikzcd}
	{(\tau_n)^iF} & {T_nF} \\
	{(\tau_n)^{i+1}F} & {\tau_nT_nF}
	\arrow[from=1-1, to=2-1]
	\arrow[from=1-1, to=1-2]
	\arrow[from=1-2, to=2-2]
	\arrow[from=2-1, to=2-2]
\end{tikzcd}\]
can always be factored through a weak equivalence, and hence the resulting right-hand side is a weak equivalence. In all known versions of Weiss calculus $T_nF$ is always $n$-polynomial, see e.g.,~\cite{We98}, \cite[Lemma 3.8]{TaggartUnitary} and \cite[Lemma 2.13]{TaggartReality}. 
\end{rem}

\begin{rem}
In all known versions of Weiss calculus the $n$-polynomial model structure may be equivalently described as the Bousfield-Friedlander localization of the projective model structure at the endofunctor 
\[
T_n: \Fun(\es{J}, \es{C}) \longrightarrow \Fun(\es{J}, \es{C})
\]
see e.g.~\cite[Proposition 6.6]{BO13}, \cite[Proposition 3.9]{TaggartUnitary}, and \cite[Proposition 2.15]{TaggartReality}.
\end{rem}

\subsection{A cofree model for the intermediate categories}

In Weiss calculus the layers of the Weiss tower are classified by spectra with an appropriate group action. This classification is through a zigzag of Quillen equivalences, and passes through what is referred to as the $n$-th intermediate category. Heuristically the $n$-th intermediate category is a model for coordinate free spectra of multiplicity $n$, i.e., spectra with structure maps given by suspension with $S^{\mathds{k}^n}$ rather than by suspension with $S^{\mathds{k}}$. This category  is a natural home for the $n$-th derivative of a functor. 

Sitting over the space of linear isometries $\es{J}(V,W)$ is the $n$-th orthogonal complement vector bundle $\gamma_n(V,W)$, with fibre over a linear isometry $f$ given by $\mathds{k}^n \otimes f(V)^\perp$.  Define the \emph{$n$-th jet category} $\es{J}_n$ to be the category with the same objects as $\es{J}$ and morphism space given by the Thom space of the $n$-th orthogonal complement vector bundle.  For the remainder of this section let $G_n$ denote the group such that spectra with a $G_n$-action classify $n$-homogeneous functors in Weiss calculus. For example, for orthogonal calculus $G_n=O(n)$, see~\cite{We95}, and for calculus with Reality, $G_n= C_2 \ltimes U(n)$, see \cite{TaggartReality}. We denote by $\es{C}[G_n]$ the category of $G_n$-objects in $\es{C}$. For $\es{C}=\T$ and $G_n=O(n)$ this is the category of $O(n)$-spaces, and for $\es{C} = C_2\T$ and $G_n= C_2 \ltimes U(n)$ this is the category of $C_2 \ltimes U(n)$-spaces.

\begin{definition}
Let $n$ be a non-negative integer. Define the \emph{$n$-th intermediate category} to be the category of $G_n$-equivariant continuous functors from $\es{J}_n$ to $\es{C}[G_n]$. We denote this category by $\Fun_{G_n}(\es{J}_n, \es{C}[G_n])$.
\end{definition}

The $n$-th intermediate category may be equipped with a stable model structure, called the $n$-stable model structure, which is a left Bousfield localization of the projective model structure. This is reminiscent of the stable model structure on spectra since the $n$-th intermediate category is equivalent to a category of modules over the commutative monoid
\[
n\bb{S} : V \longmapsto S^{\mathds{k}^n \otimes V}
\]
in the category of $G_n$-equivariant $\es{I}$-spaces, where $\es{I}$ is the category with the same objects as $\es{J}$ but only linear isometric isomorphisms, see e.g.,~\cite[Proposition 7.4]{BO13} or \cite[Proposition 4.13]{TaggartReality}.

\begin{lem}[{\cite[Proposition 7.14]{BO13}},{\cite[Theorem 4.18]{TaggartReality}}]
Let $n$ be a non-negative integer. There is a model category structure on the $n$-th intermediate category such that the cofibrations are the projective cofibrations and the fibrant objects are the $n\Omega$-spectra, i.e., those levelwise fibrant functors $X$ for which the adjoint structure maps 
\[
X(V) \to \Omega^{\mathds{k}^n \otimes W}X(V \oplus W)
\]
are weak equivalences in $\es{C}$ for all $V, W\in\es{J}_n$. 
\end{lem}

\begin{rem}
The weak equivalences of the $n$-stable model structure are the ``stable equivalences'', i.e., the local equivalences. In all known versions of Weiss calculus, the stable equivalences are precisely the $n\pi_\ast$-isomorphisms, where the $n$-homotopy groups of a functor $X:\es{J}_n^\es{C} \to \es{C}$ are given by
\[
n\pi_kX = \underset{q}{\colim}~\pi_{\dim_\R(\mathds{k}^n)q+k}X(\mathds{k}^q)
\]
see e.g.,~\cite[Proposition 7.14]{BO13}, \cite[Proposition 5.6]{TaggartUnitary}, and \cite[Proposition 5.18]{TaggartReality}.
\end{rem}

For calculus with Reality, the adjunction between free and cofree $C_2$-spaces induces a Quillen equivalence between the $n$-stable model structures.

\begin{prop}\label{prop: invariance of intermediates}
Let $n$ be a non-negative integer. The adjoint pair
\[
\adjunction{\mathds{1}}{\Fun_{G_n}(\es{J}_n, (C_ 2\ltimes U(n))\T^\sf{free})}{\Fun_{G_n}(\es{J}_n, (C_2 \times U(n))\T^\sf{cofree})}{\mathds{1}}
\]
is a Quillen equivalence.
\end{prop}
\begin{proof}
An analogous argument to that of Proposition \ref{prop: input invariant under QE} yields a Quillen equivalence between projective model structures on the $n$-th intermediate categories. To lift to a Quillen equivalence between the $n$-stable model structures, it suffices by~\cite[Theorem 3.3.20]{Hi03} to show that the derived image of the localizing set of the free $n$-stable model structure under the left adjoint is precisely the localising set for the cofree $n$-stable model structure. This follows as the localising set in both cases is
\[
\{ \es{J}_n(V \oplus W, -) \wedge S^{\mathds{k}^n \otimes W}  \longrightarrow \es{J}_n(V,-) \mid V, W \in \es{J}_n\}
\]
and the source and target are cofibrant.
\end{proof}

Let $\s^\es{U}[G_n]$ denote the category of $\es{U}$-spectra with a $G_n$-action. For orthogonal calculus this is precisely the category $\s^\mbf{O}[O(n)]$ of orthogonal spectra with a $O(n)$-action and for calculus with Reality this is the category $\s^\mbf{R}[U(n)]$ of \emph{real spectra} with a $U(n)$-action, see e.g.,~\cite[\S5]{TaggartReality}. For calculus with Reality, the adjunction between free and cofree $C_2$-spaces induces a Quillen equivalence between the categories of \emph{real spectra} with a $U(n)$-action.

\begin{prop}
Let $n$ be a non-negative integer. The adjoint pair
\[
\adjunction{\mathds{1}}{\s^\mbf{R}[U(n)]^\sf{free}}{\s^\mbf{R}[U(n)]^\sf{cofree}}{\mathds{1}}
\]
is a Quillen equivalence.
\end{prop}
\begin{proof}
By~\cite[Proposition 5.4]{TaggartReality}, we may identify the underlying category of the stable model structures as the category of $U(n)$-objects in the category of $C_2$-equivariant continuous functors $\Fun(\es{J}_1, C_2\T)$. The statement and proof are then the $n=1$ case of Proposition~\ref{prop: invariance of intermediates}.
\end{proof}

There is a well-defined continuous functor
\[
\alpha_n : \es{J}_n  \longrightarrow \es{J}_1, V \longmapsto \mathds{k}^n \otimes V
\]
and given an object $X \in \s^\es{U}[G_n]$, we let $g \in G_n$ act on $(\alpha_n)^\ast X$ by
\[
X(g \otimes V) \circ g_{X(\mathds{k}^n \otimes V)} : (\alpha_n)^\ast X(V) \longrightarrow (\alpha_n)^\ast X(V)
\]
where $X(g \otimes V)$ is the image under the functor $X$ of the \emph{internal} action of $g$ on $\mathds{k}^n \otimes V$, and $g_{X(\mathds{k}^n \otimes V)}$ is the \emph{external} action of $g$ on the $G_n$-object $X(\mathds{k}^n \otimes V)$. This action defines an equivariant continuous functor
\[
(\alpha_n)^\ast : \s^\es{U}[G_n] \longrightarrow \Fun_{G_n}(\es{J}_n, \es{C}[G_n])
\]
with left adjoint 
\[
(\alpha_n)_! : \Fun_{G_n}(\es{J}_n, \es{C}[G_n]) \to \s^\es{U}[G_n],
\]
given by (enriched) left Kan extension along $\alpha_n :\es{J}_n \to \es{J}_1$. 

\begin{prop}[{\cite[Proposition 8.3]{BO13}},{\cite[Theorem 5.2]{TaggartReality}}]
Let $n$ be a non-negative integer. The adjoint pair
\[
\adjunction{(\alpha_n)_!}{ \Fun_{G_n}(\es{J}_n, \es{C}[G_n])}{\s^\es{U}[G_n]}{(\alpha_n)^\ast}
\]
is a Quillen equivalence.
\end{prop}

\subsection{A cofree model for homogeneous functors}

In the known versions of Weiss calculus the layers of the tower, i.e., the homotopy fibres 
\[
D_nF = \hofibre(T_nF \longrightarrow T_{n-1}F)
\]
have two useful properties: firstly they are polynomial of degree less than or equal $n$, and secondly, their $(n-1)$-polynomial approximations are trivial. 

\begin{definition}
Let $n$ be a non-negative integer. A continuous functor $F: \es{J} \to \es{C}$ is \emph{homogeneous of degree $n$} or equivalently \emph{$n$-homogeneous} if it is $n$-polynomial and $T_{n-1}F$ is levelwise equivalent to the terminal object of $\es{C}$. This last property is referred to as being \emph{$n$-reduced}.
\end{definition}

There is a model structure which captures the homotopy theory of $n$-homogeneous functors. The weak equivalences are the \emph{$D_n$-equivalences}, i.e., those maps $f: E \to F$ such that $D_n(f) : D_n(E) \to D_n(F)$ are levelwise weak equivalences.

\begin{lem}[{\cite[Proposition 6.9]{BO13}},{\cite[Proposition 3.2]{TaggartReality}}]\label{lem: homogeneous model structure exists}
Let $n$ be a non-negative integer. There is a model category structure on the category of continuous functors $\Fun(\es{J}, \es{C})$ such that the fibrations are the fibrations of the $n$-polynomial model structure, and the weak equivalences are the $D_n$-equivalences. The fibrant objects are the $n$-polynomial functors, and the cofibrant objects are the projectively cofibrant $n$-reduced functors. We call this the $n$-homogeneous model structure and denote it by $\homog{n}(\es{J}, \es{C})$. 
\end{lem}

For calculus with Reality, the adjunction between free and cofree $C_2$-spaces induces a Quillen equivalence between the $n$-homogeneous model structures.

\begin{prop}\label{prop: homogeneous invariant under QE}
Let $n$ be a non-negative integer. The adjoint pair
\[
\adjunction{\mathds{1}}{\homog{n}(\es{J}, C_2\T^\sf{free})}{\homog{n}(\es{J}, C_2\T^\sf{cofree})}{\mathds{1}}
\]
is a Quillen equivalence.
\end{prop}
\begin{proof}
By~\cite[Theorem 3.3.20(2)]{Hi03}, it suffices to show that the class of colocal equivalences for the cofree model agrees with the class of colocal equivalences for the free model. Let $X \to Y$ be a $D_n$-equivalence in the free model for calculus with Reality, i.e., $D_nX \to D_nY$, is a levelwise weak equivalence in $C_2\T^\sf{free}$. Since the identity functor 
\[
\mathds{1} : C_2\T^\sf{free} \to C_2\T^\sf{cofree}
\]
preserves all weak equivalences, it follows that $D_nX \to D_nY$ is a levelwise weak equivalence in $C_2\T^\sf{cofree}$. 

For the converse, let $X \to Y$ be a $D_n$-equivalence in the cofree model for calculus with Reality. It suffices to show that given a functor $F$, the projective fibrant replacement functor $F \mapsto \fibrep_2(F)$ in $\Fun(\es{J}, C_2\T^\sf{cofree})$ is a fibrant replacement in $\Fun(\es{J}, C_2\T^\sf{free})$ since then we will have a commutative diagram
\[\begin{tikzcd}
	{D_nX} & {D_nY} \\
	{\fibrep_2D_nX} & {\fibrep_2D_nY}
	\arrow[from=1-1, to=1-2]
	\arrow[from=1-2, to=2-2]
	\arrow[from=1-1, to=2-1]
	\arrow[from=2-1, to=2-2]
\end{tikzcd}\]
in which two-out-of-three maps are weak equivalences in free projective model structure, hence the top horizontal map must also be a weak equivalence in the free projective model structure. For any functor $F$, we can cofibrantly replace the map $F \to \fibrep_2(F)$ in the free model structure, obtaining a diagram 
\[\begin{tikzcd}
	{\cofrep_1(F)} & X \\
	{\cofrep_1\fibrep_2(F)} & {\fibrep_2(X)}
	\arrow[from=1-1, to=1-2]
	\arrow[from=1-2, to=2-2]
	\arrow[from=1-1, to=2-1]
	\arrow[from=2-1, to=2-2]
\end{tikzcd}\]
in which the lower horizontal map is the derived counit of the Quillen equivalence between the free and cofree projective model structures, and $\sf{Q}_1$ denotes cofibrant replacement in the free model structure. It follows that three-out-of-four of the maps in the above square are weak equivalences in free projective model structure, hence so to is the right-hand vertical map, and the result follows. 
\end{proof}

\subsection{A cofree model for differentiation}\label{subsection: cofree differentiation}
For $m \leq n$, the inclusion $\mathds{k}^m \to \mathds{k}^n$ onto the first $m$-coordinates induces an inclusion of categories, $i_m^n \colon \es{J}_m \to \es{J}_n$. Define the \textit{restriction functor} 
\[
\res_0^n : \Fun(\es{J}_n, \es{C}) \longrightarrow \Fun(\es{J}_m, \es{C})
\]
to be precomposition with $i_m^n$, and define the \textit{induction functor} 
\[
\ind_m^n \colon \Fun(\es{J}_m, \es{C}) \longrightarrow \Fun(\es{J}_n, \es{C})
\]
to be the enriched right Kan extension along $i_m^n$. In the case $m=0$, the induction functor $\ind_0^n$ is called the \textit{$n$-th derivative}. Combining this oribits-trivial adjunction for $G_n$, see e.g.,~\cite[V.2]{MM02}, provides an adjunction
\[
\adjunction{\res_0^n/\Aut(\mathds{k}^n)}{\Fun_{G_n}(\es{J}_n, \es{C}[G_n])}{\Fun(\es{J}, \es{C})}{\ind_0^n \varepsilon^*}.
\]

\begin{thm}[{\cite[Theorem 10.1]{BO13}}, {\cite[Theorem 6.8]{TaggartReality}}]
Let $n$ be a non-negative integer. The adjoint pair
\[
\adjunction{\res_0^n/\Aut(\mathds{k}^n)}{\Fun_{G_n}(\es{J}_n, \es{C}[G_n])}{\homog{n}(\es{J}, \es{C})}{\ind_0^n\varepsilon^\ast}
\]
is a Quillen equivalence.
\end{thm}

\subsection{The equivalence of calculi}\label{subsection: equiv of calc}

Consider the commutative diagram
\begin{equation}\label{eqn: homog and intermediates}
\begin{tikzcd}
	{\homog{n}(\es{J}, C_2\T^\sf{free})} & {} & {\Fun_{G_n}(\es{J}_n, (C_2 \ltimes U(n))\T^\sf{free})} \\
	{\homog{n}(\es{J}, C_2\T^\sf{cofree})} & {} & {\Fun_{G_n}(\es{J}_n, (C_2 \ltimes U(n))\T^\sf{cofree})}
	\arrow["{\mathds{1}}"', shift right=1, from=1-3, to=2-3]
	\arrow["{\mathds{1}}"', shift right=1, from=2-3, to=1-3]
	\arrow["{\mathds{1}}"', shift right=1, from=1-1, to=2-1]
	\arrow["{\mathds{1}}"', shift right=1, from=2-1, to=1-1]
	\arrow["{\res_0^n/U(n)}"', shift right=1, from=1-3, to=1-1]
	\arrow["{\res_0^n/U(n)}"', shift right=1, from=2-3, to=2-1]
	\arrow["{\ind_0^n\varepsilon^\ast}"', shift right=1, from=1-1, to=1-3]
	\arrow["{\ind_0^n\varepsilon^\ast}"', shift right=1, from=2-1, to=2-3]
\end{tikzcd}
\end{equation}

In Weiss calculus the largest task in setting up the theory is demonstrating that the restriction-induction adjunction is an equivalence between the underlying $\infty$-categories of the $n$-homogeneous model structure and the $n$-stable model structure. The diagram~\eqref{eqn: homog and intermediates}, implies this result for cofree calculus with Reality if it is true for free calculus with Reality.

\begin{thm}\label{thm: differentiation invariant under QE}
Let $n$ be a non-negative integer. The adjunction
\[
\adjunction{\res_0^n/U(n)}{\homog{n}(\es{J}, C_2\T^\sf{cofree})}{\Fun_{G_n}(\es{J}_n, (C_2 \ltimes U(n))\T^\sf{cofree})}{\ind_0^n\varepsilon^\ast}
\]
is a Quillen equivalence.
\end{thm}
\begin{proof}
Diagram \eqref{eqn: homog and intermediates} is a diagram in which three-out-of-four adjoint pairs are Quillen equivalences. By the two-out-of-three property for Quillen equivalences, the result will follow if the lower horizontal adjunction is a Quillen adjunction\footnote{Note that Quillen adjunctions do not satisfy a two-out-of-three property, see e.g.,~\cite[Example 4.1.10]{BalchinMCat} hence we cannot use the same tactic to prove that the lower horizontal adjunction is a Quillen adjunction}. The required Quillen adjunction follows readily from the standard techniques, see e.g.,~\cite[Lemmas 9.1, 9.2 and 9.4]{BO13}.  In particular, it suffices to show that if $F$ is fibrant in the $n$-homogeneous model structure then $\ind_0^n\varepsilon^\ast(F)$ is an $n\Omega$-spectrum in the cofree $n$-th intermediate category, and that the right adjoint applied to a weak equivalence between levelwise fibrant $n$-polynomial functors is a weak equivalence in the $n$-stable model structure. The latter is a direct consequence of the formation of the $n$-homogeneous model structure as a right Bousfield localization of the $n$-polynomial model structure, see Lemma~\ref{lem: homogeneous model structure exists}. The former follows from~\cite[Corollary 5.12]{We95}.
\end{proof}

The other key aspect of Weiss calculus is the equivalence between the intermediate category and some category of spectra with a group action. This equivalence allows us to view the derivative in terms of a single spectrum and forms one part of the classification theorem for homogeneous functors. In calculus with Reality this is a two-step process. Similar analysis as for Diagram~\eqref{eqn: homog and intermediates} applied to the diagram

\begin{equation}\label{eqn: intermediates and spectra}
\begin{tikzcd}
	{\Fun_{G_n}(\es{J}_n, (C_2\ltimes U(n))\T^\sf{free})} & {} & {\s^\mbf{R}[U(n)]^\sf{free}} \\
	{\Fun_{G_n}(\es{J}_n, (C_2 \ltimes U(n))\T^\sf{cofree})} & {} & {\s^\mbf{R}[U(n)]^\sf{cofree}}
	\arrow["{\mathds{1}}"', shift right=1, from=1-1, to=2-1]
	\arrow["{\mathds{1}}"', shift right=1, from=2-1, to=1-1]
	\arrow["{(\alpha_n^\mbf{R})_!}", shift left=1, from=2-1, to=2-3]
	\arrow["{(\alpha_n^\mbf{R})^\ast}", shift left=1, from=2-3, to=2-1]
	\arrow["{(\alpha_n^\mbf{R})^\ast}", shift left=1, from=1-3, to=1-1]
	\arrow["{(\alpha_n^\mbf{R})_!}", shift left=1, from=1-1, to=1-3]
	\arrow["{\mathds{1}}"', shift right=1, from=1-3, to=2-3]
	\arrow["{\mathds{1}}"', shift right=1, from=2-3, to=1-3]
\end{tikzcd}
\end{equation}
implies that the classification of homogeneous functors in terms of \emph{real} spectra holds in the cofree model if and only if it holds in the free model.

\begin{thm}\label{thm: intermediate and spectra invariant}
Let $n$ be a non-negative integer. The adjunction
\[
\adjunction{(\alpha_n^\mbf{R})_!}{\Fun_{U(n)}(\es{J}_n, (C_2 \ltimes U(n))\T^\sf{cofree})}{\s^\mbf{R}[U(n)]^\sf{cofree}}{(\alpha_n^\mbf{R})^\ast}
\]
is a Quillen equivalence.
\end{thm}
\begin{proof}
The required Quillen adjunction follows readily from the standard techniques, see e.g.,~\cite[Proposition 8.3]{BO13}. The Quillen equivalence follows by two-out-of-three, using the Quillen equivalence of~\cite[Theorem 6.2]{TaggartReality}.
\end{proof}

\subsection{Calculus with Reality on free and cofree $C_2$-spaces}

In calculus with Reality, the classification of $n$-homogeneous functors has one extra step, and is given by the zigzag
\begin{equation}\label{eqn: Reality zigzag}
\begin{tikzcd}
	{\homog{n}(\es{J}_0^\mbf{R},C_2\T^\sf{free})} \\
	{\Fun_{C_2 \ltimes U(n)}(\es{J}_n^\mbf{R}, (C_2\ltimes U(n)\T^\sf{free})} \\
	{\s^\mbf{R}[U(n)]^\sf{free}} \\
	{\s^\mbf{O}[C_2 \ltimes U(n)]^\sf{free}}
	\arrow["{(\alpha_n^\mbf{R})^\ast}"', shift right=1, from=3-1, to=2-1]
	\arrow["{(\alpha_n^\mbf{R})_!}"', shift right=1, from=2-1, to=3-1]
	\arrow["\psi", shift left=2, from=3-1, to=4-1]
	\arrow["{L_\psi}", from=4-1, to=3-1]
	\arrow["{\res_0^n/U(n)}", shift left=1, from=2-1, to=1-1]
	\arrow["{\ind_0^n\varepsilon^\ast}", shift left=1, from=1-1, to=2-1]
\end{tikzcd}
\end{equation}
where $\s^\mbf{O}[C_2 \ltimes U(n)]$ is the category of orthogonal spectra with an action of $C_2 \ltimes U(n)$. The adjoint pair $(L_\psi, \psi)$ has right adjoint given by
\[
\psi(X)(V) = \Map_\ast(S^{iV}, X(V_\C))
\]
and left adjoint given by an appropriate enriched coend formula, see e.g.~\cite[Proposition 5.5]{TaggartReality}. Note that this last adjunction is a Quillen equivalence by~\cite[Theorem 5.8]{TaggartReality}. To replace the zigzag of free model structures by a zigzag of cofree model structures it suffices to show that the diagram
\begin{equation}\label{eqn: spectra diagram}
\begin{tikzcd}
	{\s^\mbf{R}[U(n)]^\sf{free}} & {\s^\mbf{R}[U(n)]^\sf{cofree}} \\
	{\s^\mbf{O}[C_2 \ltimes U(n)]^\sf{free}} & {\s^\mbf{O}[C_2 \ltimes U(n)]^\sf{cofree}}
	\arrow["{\mathds{1}}", shift left=1, from=1-1, to=1-2]
	\arrow["{\mathds{1}}", shift left=1, from=1-2, to=1-1]
	\arrow["{\mathds{1}}", shift left=1, from=2-1, to=2-2]
	\arrow["{\mathds{1}}", shift left=1, from=2-2, to=2-1]
	\arrow["{L_\psi}"', shift right=1, tail reversed, no head, from=1-1, to=2-1]
	\arrow["\psi", shift left=1, from=1-1, to=2-1]
	\arrow["\psi", shift left=1, from=1-2, to=2-2]
	\arrow["{L_\psi}", shift left=1, from=2-2, to=1-2]
\end{tikzcd}
\end{equation}
is a commutative diagram of Quillen equivalences. This will follow from the following result.

\begin{prop}
Let $n$ be a non-negative integer. The adjoint pair
\[
\adjunction{L_\psi}{\s^\mbf{O}[C_2 \ltimes U(n)]^\sf{cofree}}{\s^\mbf{R}[U(n)]^\sf{cofree}}{\psi}
\]
is a Quillen equivalence.
\end{prop}
\begin{proof}
Since Diagram~\eqref{eqn: spectra diagram} clearly commutes on the level of $(1-)$categories, it suffices to exhibit that the adjunction in question is a Quillen adjunction, then the diagram will be a diagram of Quillen adjunctions in which three-out-of-four are Quillen equivalences, hence the fourth adjoint pair must also be a Quillen equivalence. To see that the adjunction is a Quillen adjunction note that by construction the right adjoint preserves acyclic fibrations as acyclic fibrations are precisely the levelwise acyclic fibrations. It hence suffices to show that the right adjoint preserves fibrant objects. The proof of~\cite[Proposition 5.6]{TaggartReality} demonstrates that the right adjoint sends real $\Omega$-spectra to orthogonal $\Omega$-spectra, hence it suffices to show that if $X \in \s^\mbf{R}[U(n)]$ is levelwise cofree that $\psi X$ is levelwise cofree. This follows by definition and the interaction of loops and mapping spaces.
\end{proof}

\section{Polynomial functors}

\subsection{Recovering input functors}\label{section: input functors}
 Using the invariance of Weiss calculus under Quillen equivalence, we shall base our comparison between calculus with Reality and orthogonal calculus on the cofree model for calculus with Reality. To ease exposition we will neglect the superscripts for model categories when it will not cause confusion.

The complexification of real vector spaces induces a functor $c: \es{J}^\mbf{O} \to \es{J}^\mbf{R}$, and hence a Quillen adjunction
\[
\adjunction{c_!}{\Fun(\es{J}^\mbf{O}_0, \T)}{\Fun(\es{J}^\mbf{R}_0, \T)}{c^\ast}
\]
on the level of enriched functor categories, see e.g.,~\cite[Lemma 3.2]{TaggartOCandUC}. On the other hand, the adjunction
\[
\adjunction{i^\ast}{\T}{C_2\T^\sf{fine}}{(-)^{C_2}}
\]
extends to an adjunction between spaces and cofree $C_2$-spaces
\[
\adjunction{i^\ast}{\T}{C_2\T^\sf{cofree}}{(-)^{C_2}}
\]
since cofree $C_2$-spaces are a left Bousfield localization of the fine model structure. This lifts to a Quillen adjunction on the level of model structures
\[
\adjunction{i^\ast}{\Fun(\es{J}^\mbf{R}_0,\T)}{\Fun(\es{J}^\mbf{R}_0,C_2\T^\sf{cofree})}{(-)^{C_2}}
\]
with both categories equipped with the projective model structure. It should be noted that on the left-hand side we have a category of $\T$-enriched functors, while on the right-hand side we have a category of $C_2\T$-enriched functors. This difference in enrichment does not interfere with the existence of the adjunction. The result is a composite Quillen adjunction
\[
\adjunction{i^\ast c_!}{\Fun(\es{J}_0^\mbf{O}, \T)}{\Fun(\es{J}_0^\mbf{R}, C_2\T^\sf{cofree})}{c^*(-)^{C_2}}
\]
through which the right adjoint sends a functor with Reality to an orthogonal functor. Given a functor $F$ in orthogonal calculus, the image of $F$ under the unit of the adjunction recovers $F$ up to a shift in degree of the vector spaces since the left Kan extension $c_!F$ is naturally isomorphism (by enriched Yoneda) to the functor $\R^{k} \longmapsto F(\R^{2k})$.  In particular this means that we cannot expect to completely recover orthogonal calculus from calculus with Reality.

\subsection{Polynomial functors} We begin our comparison with the polynomial functors. Unlike many other comparisons between calculi, see e.g.,~\cite{BE16, TaggartOCandUC} we do not need to invoke the classification theorem for homogeneous functors in order to compare polynomial functors. In both versions of Weiss calculus the $n$-polynomial functors are defined via homotopy limits indexed on topological categories, see Definition~\ref{def: n-poly}. In the orthogonal calculus the indexing poset $\mbf{O}_{\leq n+1}$ is the poset of non-zero inner product subspaces of $\R^{n+1}$, and in calculus with Reality the indexing poset $\mbf{R}_{\leq n+1}$ is the poset of non-zero inner product subspaces of $\C \otimes \R^{n+1}$ of the form $\C \otimes U =: U_\C$, for $U \in \mbf{O}_{\leq n+1}$. The comparison of $n$-polynomial functors boils down to a (co)finality statement of homotopy (co)limits indexed on such topological categories, see e.g.,~\cite[Appendix]{Li09}. We will use that (co)finality indexed on topological categories can be proven by showing that an appropriate comma category as a \emph{topologically terminal object}, that is an object $\ast \in \es{C}$ which is terminal in the underlying category, and the unique map of spaces $!: \sf{Ob}(\es{C}) \to \sf{Mor}(\es{C})$ sending every object $c \in \es{C}$ to the unique map $c\to\ast$ is continuous, see e.g.,~\cite[Remark A.6]{Li09}. 

\begin{thm}\label{thm: polynomials preserved}
Let $n$ be a non-negative integer.
If $F$ is a fibrant $n$-polynomial functor in the cofree model for calculus with Reality, then the orthogonal functor $c^\ast(F)^{C_2}$ is $n$-polynomial in orthogonal calculus. 	
\end{thm}
\begin{proof}
Since $F$ is levelwise fibrant in the cofree model structure, we can replace $C_2$-fixed points by homotopy $C_2$-fixed points. By~\cite[Definition 5.1]{We95}, we have to show that for each $V \in \es{J}^\mbf{O}$, the canonical map
\[
\rho_\mbf{O}: F(V_\C)^{hC_2} \longrightarrow \underset{U \in \mbf{O}_{\leq n+1}}{\holim}~F(V_\C \oplus U_{\C})^{hC_2}
\]
is a weak equivalence. This map factors as the composite
\begin{equation*}\label{eqn: commuting diagram for n-poly}
\begin{tikzcd}[ampersand replacement=\&]
	{F(V_\C)^{hC_2}} \& {\underset{U \in \mbf{O}_{\leq n+1}}{\holim}~F(V_\C \oplus U_{\C})^{hC_2}} \\
	\\
	{\left[\underset{U_\C \in \mbf{R}_{\leq n+1}}{\holim}~F(V_{\C} \oplus U_{\C})\right]^{hC_2}} \& {\left[\underset{U \in \mbf{O}_{\leq n+1}}{\holim}~F(V_{\C} \oplus U_{\C})\right]^{hC_2}}
	\arrow["{(\rho_{\mbf{R}})^{hC_2}}"', from=1-1, to=3-1]
	\arrow["{c^\ast}"', from=3-1, to=3-2]
	\arrow["{\rho_{\mbf{O}}}", from=1-1, to=1-2]
	\arrow["\simeq"', from=3-2, to=1-2]
\end{tikzcd}
\end{equation*}
where the right-hand vertical map follows from the universal property of homotopy limits and is an equivalence  since right derived functors commute with homotopy limits, and the lower horizontal map is the induced map from the map of topological posets $c: \mbf{O}_{\leq n+1} \to \mbf{R}_{\leq n+1}$. It suffices to show that this lower horizontal map is a weak equivalence, since by assumption the left-hand vertical map is a weak equivalence. By the dual of~\cite[Lemma A.5]{Li09}, it suffices to prove that for every $V_\C \in \mbf{R}_{\leq n+1}$ the comma category $(c \downarrow V_\C)$ has a topologically terminal object.

The comma category $(c \downarrow V_\C)$ has objects pairs $(U, f_U)$ with $U \in \mbf{O}_{\leq n+1}$ and $f_U: U_\C \to V_\C$ in $\mbf{R}_{\leq n+1}$. A morphism $h: (U, f_U) \to (W,f_W)$ in $(c \downarrow V_\C)$ is the data of a linear isometry $h: U \to W$ in $\mbf{O}_{\leq n+1}$ such that the triangle
\[\begin{tikzcd}
	{U_\C} && {W_\C} \\
	& {V_\C}
	\arrow["{f_U}"', from=1-1, to=2-2]
	\arrow["{f_U}", from=1-3, to=2-2]
	\arrow["{c(h)}", from=1-1, to=1-3]
\end{tikzcd}\]
commutes in $\mbf{R}_{\leq n+1}$. Since $\mbf{R}_{\leq n+1}$ is the poset of subspaces of $\C \otimes \R^{n+1}$ of the form $\C \otimes U$ for $U$ a subspace of $\R^{n+1}$ ordered by inclusion, it follows that every map $U_\C \to V_\C$ in $\mbf{R}_{\leq n+1}$ must be either the identity on $U_\C$ or the inclusion of a complex subspace. One can readily check that the inclusion of a complex subspace is fixed under the $C_2$-action by complex conjugation hence is of the form $c(f)$ for $g$ a real linear map. It follows that every map $f: U_\C \to V_\C$ must be of the form $c(g)$ for $g: U \to V$ in $\mbf{O}_{\leq n+1}$.

The comma category $(c \downarrow V_\C)$ is a category object in spaces with topology inherited from the product topology on $\sf{Ob}(\mbf{O}_{\leq n+1}) \times \sf{Mor}(\mbf{R}_{\leq n+1})$. On the underlying non-topological category, the object $(V, \id_{V_\C})$ is terminal in $(c \downarrow V_\C)$, with the unique map $(U, f_U) \to (V, \id_{V_\C})$ given by $f_U$, using the above fact that $f_U = c(g)$ for some $g: U \to V$ in $\mbf{O}_{\leq n+1}$. It is left to show that the unique map 
\[
! : \sf{Ob}(c \downarrow V_\C) \longrightarrow \sf{Mor}(c \downarrow V_\C),
\]
is continuous. To see this, note that the unique map is the restriction of domain and codomain of a the projection map
\[
\pi_2: \sf{Ob}(\mbf{O}_{\leq n+1})\times \sf{Mor}(\mbf{R}_{\leq n+1}) \longrightarrow \sf{Mor}(\mbf{R}_{\leq n+1}),
\]
hence continuous. 
\end{proof}

By examining the construction of the $n$-polynomial approximation $T_n$ as iterates of the functor $\tau_n$, where 
\[
\tau_nF(V) = \underset{0 \neq U \subseteq \mathds{k}^{n+1}}{\holim}~F(V \oplus U),
\]
the above result will provide the interaction between the polynomial approximations in orthogonal calculus and calculus with Reality. The definition of $\tau_n$ implies the following corollary.

\begin{cor}\label{cor: C_2-fixed points preserve tau}
Let $n$ be a non-negative integer. If $F$ is an fibrant functor in the cofree model for calculus with Reality, then there is a levelwise equivalence of orthogonal functors
\[
c^\ast(\tau_n^\mbf{R}F)^{C_2} \longrightarrow \tau_n^\mbf{O}(c^\ast F^{C_2}).
\]
\end{cor}
\begin{proof}
The above proof provides an equivalence on homotopy $C_2$-fixed points, and the cofree assumption on $F$ concludes that all maps from $C_2$-fixed points to homotopy $C_2$-fixed points are weak equivalences.
\end{proof}

\subsection{Polynomial approximation}

To give an equivalence on polynomial approximations, we first provide an equivalence on iterations of $\tau_n$, used in constructing the polynomial approximation functor.

\begin{lem}\label{lem: C_2-fixed points preserve powers of tau}
Let $k$ and $n$ be non-negative integers. If $F$ is an fibrant functor in the cofree model for calculus with Reality, then there is a levelwise equivalence of orthogonal functors
\[
c^\ast((\tau_n^\mbf{R})^kF)^{C_2} \longrightarrow (\tau_n^\mbf{O})^k(c^\ast F^{C_2}).
\]		
\end{lem}
\begin{proof}
This follows as in Corollary~\ref{cor: C_2-fixed points preserve tau} by writing the iterated homotopy limit $(\tau_n)^k$	 as a single homotopy limit, and noting that our fibrant assumption insures homotopy invariance. 
\end{proof}

Taking the homotopy colimit of Lemma~\ref{lem: C_2-fixed points preserve powers of tau} over $k$ provides an equivalence on polynomial approximations. 

\begin{thm}\label{thm: polynomial approximation preserved}
Let $n$ be a non-negative integer. If $F$ is an fibrant functor in the cofree model for calculus with Reality, then there is a levelwise equivalence of orthogonal functors
\[
c^\ast(T_n^\mbf{R}F)^{C_2} \longrightarrow T_n^\mbf{O}(c^\ast F^{C_2}).
\]
\end{thm}
\begin{proof}
Inductively use Lemma~\ref{lem: C_2-fixed points preserve powers of tau} with the definition of $T_n$, and the fact that $C_2$-fixed points commute with sequential homotopy colimits. 
\end{proof}


With this we can show that on the level of homotopy theories (or $\infty$-categories or model categories) the adjunction between functor categories lifts to an equivalence between the $\infty$-categories of $n$-polynomial functors in the calculi. 

\begin{prop}\label{prop: QA for poly model structures}
Let $n$ be a non-negative integer. The adjoint pair
\[
\adjunction{i^\ast c_!}{\poly{n}(\es{J}^\mbf{O}, \T)}{\poly{n}(\es{J}^\mbf{R}, C_2\T)}{c^\ast(-)^{C_2}}
\]
is a Quillen adjunction.
\end{prop}
\begin{proof}
The acyclic fibrations of the $n$-polynomial model structure from calculus with Reality are the levelwise acyclic fibrations which the right adjoint preserves by the discussion in Subsection~\ref{section: input functors}. 

Moreover, a map $f: E \to F$ is a fibration in the $n$-polynomial model structure from calculus with Reality if and only if $f: E  \to F$ is a underlying fibration (which are preserved by the right adjoint), and the canonical square
\[\begin{tikzcd}
	E & F \\
	{T_n^\mbf{R}(E)} & {T_n^\mbf{R}(F)}
	\arrow[from=1-1, to=2-1]
	\arrow["f", from=1-1, to=1-2]
	\arrow[from=1-2, to=2-2]
	\arrow["{T_n^\mbf{R}(f)}"', from=2-1, to=2-2]
\end{tikzcd}\]
is a homotopy pullback in the projective model structure. As the right adjoint is a right Quillen functor on projective model structures it preserves homotopy pullbacks and hence we see that we have a diagram
\[\begin{tikzcd}
	{c^\ast E^{C_2}} & {} & {c^\ast F^{C_2}} \\
	{c^\ast(T_n^\mbf{R}(E))^{C_2}} & {} & {c^\ast(T_n^\mbf{R}(F))^{C_2}} \\
	{T_n^\mbf{O}(c^\ast E^{C_2})} & {} & {T_n^\mbf{O}(c^\ast F^{C_2})}
	\arrow[from=1-3, to=2-3]
	\arrow["\simeq", from=2-3, to=3-3]
	\arrow["{c^\ast f^{C_2}}", from=1-1, to=1-3]
	\arrow["{c^\ast(T_n^\mbf{R}(f))^{C_2}}", from=2-1, to=2-3]
	\arrow["{T_n^\mbf{O}(c^*f^{C_2})}"', from=3-1, to=3-3]
	\arrow["\simeq"', from=2-1, to=3-1]
	\arrow[from=1-1, to=2-1]
\end{tikzcd}\]
in which the top square is a homotopy pullback as the lower vertical arrows are weak equivalences, and hence the outer square must also be a homotopy pullback. It follows that the right adjoint preserves fibrations.
\end{proof}

\section{The Weiss tower}

\subsection{Homogeneous functors} The next step in understanding the relationship between Weiss towers is understanding how the layers of the towers interact. More generally, we obtain the following relationship between homogeneous functors. 

\begin{thm}
Let $n$ be a non-negative integer. If $F$ is a homogeneous functor of degree $n$ in the cofree model for calculus with Reality, then $c^\ast F^{C_2}$ is a homogeneous functor of degree $n$ in orthogonal calculus. 
\end{thm}
\begin{proof}
A homogeneous functor of degree $n$ is precisely a polynomial functor of degree less than or equal $n$, such that the $(n-1)$-polynomial approximation vanishes. In light of Theorem \ref{thm: polynomials preserved}, it suffices to show that triviality of $(n-1)$-polynomial approximations is preserved, which follows from Theorem \ref{thm: polynomial approximation preserved}.
\end{proof}

As with the $\infty$-categories of polynomial functors, the adjoint pair between functor categories induces an adjoint pair between $\infty$-categories of homogeneous functors.  In the following proof we make use of the fact that many of the results in unitary calculus hold also in calculus with Reality. The proofs of the results uses in this way can be verbatim translated to calculus with Reality since they only rely on formal calculus arguments and the classification of homogeneous functors. 

\begin{prop}\label{prop: QA for homog model structures}
Let $n$ be a non-negative integer. The adjoint pair
\[
\adjunction{i^\ast c_!}{\homog{n}(\es{J}^\mbf{O}, \T)}{\homog{n}(\es{J}^\mbf{R}, C_2\T^\sf{cofree})}{c^\ast(-)^{C_2}}
\]
is a Quillen adjunction.
\end{prop}
\begin{proof}
The right adjoint preserves fibrations since the fibrations of the $n$-homogeneous model structure coincide with the fibrations of the $n$-polynomial model structure, which are preserved by Proposition \ref{prop: QA for poly model structures}. By (the calculus with Reality version of)~\cite[Proposition 8.3]{TaggartUnitary} the acyclic fibrations of the $n$-homogeneous model structure are the fibrations of the $(n-1)$-polynomial model structure which are also $D_n$-equivalences. Since fibrations of the $(n-1)$-polynomial model structure are preserved, to complete the proof it suffices to prove that the right adjoint is homotopical and hence that the weak equivalences are reflected. This follows from (the calculus of Reality version of)~\cite[Proposition 8.2]{TaggartUnitary}, Theorem \ref{thm: polynomial approximation preserved}, and that right adjoints commute with right Kan extensions.
\end{proof}

Using that $D_nF$ is the universal approximation to $F$ in the underlying $\infty$-category of $n$-homogeneous functors we obtain the following relationship between $n$-homogeneous approximations, or $n$-th layers of Weiss towers. 

\begin{cor}\label{cor: layers agree}
Let $n$ be a non-negative integer. If $F$ is an fibrant functor in the cofree model for calculus with Reality, then there is a levelwise equivalence of orthogonal functors
\[
c^\ast(D_n^\mbf{R}F)^{C_2} \longrightarrow D_n^\mbf{O}(c^\ast F^{C_2}).
\]
\end{cor}

\subsection{Weiss towers}

We gather together the results obtained so far to give the relationship between Weiss towers in orthogonal calculus and Weiss towers in calculus with Reality. Since the Weiss tower of a functor $F$ is a cofiltered object in the category of input functors, one can view the Weiss tower as a functor 
\[
\sf{Tow} : \Fun(\es{J}, \es{C}) \longrightarrow \sf{CoFil}(\Fun(\es{J}, \es{C}))
\]
where $\es{C}$ is either $\T$ or $C_2\T$ for our purposes.

\begin{thm}
If $F$ is a fibrant functor in the cofree model for calculus with Reality, then there is a levelwise weak equivalence between the image of the Weiss tower of $F$ in calculus with Reality under the functor $c^\ast(-)^{C_2}$ and the Weiss tower of the functor $c^\ast F^{C_2}$ in orthogonal calculus, i.e.,
\[
c^\ast \sf{Tow}^\mbf{R}(F)^{C_2} \longrightarrow \sf{Tow}^\mbf{O}(c^\ast F^{C_2}).
\]
\end{thm}
\begin{proof}
By Theorem~\ref{thm: polynomial approximation preserved}, the $n$-th term in the Weiss towers agree up to levelwise weak equivalence and by Corollary~\ref{cor: layers agree}, the homotopy fibres of the maps between the $n$-th approximations and $(n-1)$-st approximations agree up to levelwise weak equivalence, hence the towers agree up to levelwise weak equivalence. In particular, we obtain a commutative diagram
\[\begin{tikzcd}
	{c^\ast(D_n^\mbf{R}F)^{C_2}} & {D_n^\mbf{O}(c^\ast F^{C_2})} \\
	{c^\ast(T_n^\mbf{R}F)^{C_2}} & {T_n^\mbf{O}(c^\ast F^{C_2})} \\
	{c^\ast(T_{n-1}^\mbf{R}F)^{C_2}} & {T_{n-1}^\mbf{O}(c^\ast F^{C_2})}
	\arrow[from=1-1, to=1-2]
	\arrow[from=1-1, to=2-1]
	\arrow[from=2-1, to=3-1]
	\arrow[from=1-2, to=2-2]
	\arrow[from=2-2, to=3-2]
	\arrow[from=3-1, to=3-2]
	\arrow[from=2-1, to=2-2]
\end{tikzcd}\]
in which the columns are homotopy fibre sequences and the horizontal arrows are levelwise weak equivalences.
\end{proof}

\section{Higher categorical analysis}\label{section: model categories}

Model categories for functor calculus have proved a useful tool in computations. We have seen that on the point-set level orthogonal calculus is recoverable from calculus with Reality. In this section, we complete the comparison by considering the model structures which appear in the model categorical classification of homogeneous functors, or more precisely the underlying $\infty$-categories of these model categories. Figure~\ref{fig: model categories} displays the relationship between the model categories for orthogonal calculus and the model categories for calculus with Reality. In this section we explore in which sense this diagram is commutative\footnote{We say that a diagram of adjoint pairs is commutative if the respective diagram of right adjoints commutes. One can easily show that this is equivalent to the respective diagram of left adjoints commuting.} on the level of underlying $\infty$-categories.
\begin{figure}[ht]
\[\begin{tikzcd}
	{\s^\mbf{O}[C_2 \ltimes U(n)]} &&& {\s^\mbf{O}[O(n)]} \\
	{\s^\mbf{R}[U(n)]} \\
	{\Fun_{C_2 \ltimes U(n)}(\es{J}_n^\mbf{R}, (C_2 \ltimes U(n))\T)} &&& {\Fun_{O(n)}(\es{J}_n^\mbf{O}, O(n)\T)} \\
	\\
	{\homog{n}(\es{J}_0^\mbf{R}, C_2\T)} &&& {\homog{n}(\es{J}_0^\mbf{O},\T)} \\
	\\
	{\poly{n}(\es{J}_0^\mbf{R}, C_2\T)} &&& {\poly{n}(\es{J}_0^\mbf{O},\T)} \\
	\\
	{\poly{n-1}(\es{J}_0^\mbf{R}, C_2\T)} &&& {\poly{n-1}(\es{J}_0^\mbf{O},\T)} \\
	\\
	{\Fun_{C_2}(\es{J}_0^\mbf{R}, C_2\T)} &&& {\Fun(\es{J}_0^\mbf{O}, \T)} \\
	&&& {}
	\arrow["{c^\ast(-)^{C_2}}"', shift right=1, from=11-1, to=11-4]
	\arrow["{c^\ast(-)^{C_2}}"', shift right=1, from=9-1, to=9-4]
	\arrow["{c^\ast(-)^{C_2}}"', shift right=1, from=7-1, to=7-4]
	\arrow["{c^\ast(-)^{C_2}}"', shift right=1, from=5-1, to=5-4]
	\arrow["{\psi_n (-)^{C_2}}"', shift right=1, from=3-1, to=3-4]
	\arrow["{(-)^{C_2}}"', shift right=1, from=1-1, to=1-4]
	\arrow["{(C_2 \ltimes U(n))_+ \wedge_{C_2 \times O(n)}(-)}"', shift right=1, from=1-4, to=1-1]
	\arrow["{(C_2 \ltimes U(n)_+ \wedge_{C_2 \times O(n)}(L_{\psi_n}(-))}"', shift right=1, from=3-4, to=3-1]
	\arrow["{i^\ast c_!}"', shift right=1, from=5-4, to=5-1]
	\arrow["{i^\ast c_!}"', shift right=1, from=7-4, to=7-1]
	\arrow["{i^\ast c_!}"', shift right=1, from=9-4, to=9-1]
	\arrow["{i^\ast c_!}"', shift right=1, from=11-4, to=11-1]
	\arrow["{\mathds{1}}", shift left=1, from=9-1, to=11-1]
	\arrow["{\mathds{1}}"', shift right=1, from=9-1, to=7-1]
	\arrow["{\mathds{1}}"', shift right=2, from=7-1, to=5-1]
	\arrow["{\mathds{1}}"', from=5-1, to=7-1]
	\arrow["{\mathds{1}}"', shift right=1, from=7-1, to=9-1]
	\arrow["{\mathds{1}}", shift left=1, from=11-1, to=9-1]
	\arrow["{\mathds{1}}", shift left=1, from=9-4, to=11-4]
	\arrow["{\mathds{1}}"', shift right=1, from=9-4, to=7-4]
	\arrow["{\mathds{1}}"', shift right=1, from=7-4, to=5-4]
	\arrow["{\mathds{1}}"', shift right=1, from=5-4, to=7-4]
	\arrow["{\mathds{1}}"', shift right=1, from=7-4, to=9-4]
	\arrow["{\mathds{1}}", shift left=1, from=11-4, to=9-4]
	\arrow["{\ind_0^n\varepsilon^\ast}"', shift right=2, from=5-1, to=3-1]
	\arrow["{\ind_0^n\varepsilon^\ast}"', shift right=1, from=5-4, to=3-4]
	\arrow["{\res_0^n/U(n)}"', shift right=1, from=3-1, to=5-1]
	\arrow["{\res_0^n/O(n)}"', shift right=1, from=3-4, to=5-4]
	\arrow["{(\alpha_n^\mbf{R})^\ast}", shift left=1, from=2-1, to=3-1]
	\arrow["{(\alpha_n^\mbf{R})_!}", shift left=1, from=3-1, to=2-1]
	\arrow["\psi"', shift right=1, from=2-1, to=1-1]
	\arrow["{L_\psi}"', shift right=1, from=1-1, to=2-1]
	\arrow["{(\alpha_n^\mbf{O})^\ast}", shift left=1, from=1-4, to=3-4]
	\arrow["{(\alpha_n^\mbf{O})_!}", shift left=1, from=3-4, to=1-4]
	\arrow["{\circled{1}}"{description}, draw=none, from=11-1, to=9-4]
	\arrow["{\circled{2}}"{description}, draw=none, from=9-1, to=7-4]
	\arrow["{\circled{3}}"{description}, draw=none, from=7-1, to=5-4]
	\arrow["{\circled{4}}"{description}, draw=none, from=5-1, to=3-4]
	\arrow["{\circled{5}}"{description}, draw=none, from=3-1, to=1-4]
\end{tikzcd}\]
\caption{Model categories for calculus with Reality and orthogonal calculus.}
\label{fig: model categories}
\end{figure}

Observe that by the model categorical work of Barnes and Oman~\cite{BO13}, the right-hand column consists of Quillen adjunctions, and by~\cite{TaggartReality} the left-hand column also consists of Quillen adjunctions, hence both columns may be suitably derived to give adjunctions of $\infty$-categories. By Section \ref{section: input functors}, Proposition~\ref{prop: QA for poly model structures}, and Proposition~\ref{prop: QA for homog model structures}, we have shown that the lower three horizontal adjunctions induce adjunctions of $\infty$-categories, and the lower three squares ($\circled{1}-\circled{3}$) are clearly commutative. To complete our comparisons it suffices to exhibit adjunctions between the respective $\infty$-categories of spectra and the intermediate categories, and show that the remaining square ($\circled{4}$) and pentagon ($\circled{5}$) are commutative on the $\infty$-categorical level. It is essential that we use $\infty$-categories here as the remaining squares only commute after suitable (co)fibrant replacements.

\subsection{The adjunctions} For the categories of spectra, we construct an adjunction of $\infty$-categories using classical stable homotopy theory.

\begin{lem}
Let $n$ be a non-negative integer. There is a Quillen adjunction
\[
\adjunction{(C_2 \ltimes U(n))_+ \wedge_{C_2 \times O(n)} (-)}{\s^\mbf{O}[O(n)]}{\s^\mbf{O}[C_2\ltimes U(n)]}{(-)^{C_2}}.
\]
\end{lem}
\begin{proof}
By a standard equivariant stable homotopy theory argument, for a subgroup $H$ of a compact Lie group $G$, there is a Quillen adjunction
\[
\adjunction{G_+ \wedge_{N_G(H)} (-)}{\s^\mbf{O}[W_{G}(H)]}{\s^\mbf{O}[G]}{(-)^{H}},
\]
between stable model structures. It hence suffices to show that the normaliser of $C_2$ in $C_2 \ltimes U(n)$ is $C_2 \times O(n)$ and that the associated Weyl group is $O(n)$. By examining the normaliser condition one sees that the normaliser of $C_2$ is $C_2 \ltimes U(n)$ is
\[
C_2 \ltimes \{ M \in U(n) \mid  MM^T = I_n\}
\]
We claim that this is precisely $C_2 \times O(n)$. To see this, note that If $M$ is as above, then $M \in U(n)$ implies $g \cdot M^T = M^\ast = M^{-1}$ for $g \in C_2$ the non-identity element. In particular, $g \cdot M^T = M^T$, i.e., $M^T=M^\ast$ and hence $M$ must be a real matrix. It follows that $M \in O(n)$. An analogous converse argument yields the identification 
\[
C_2 \ltimes \{ M \in U(n) \mid  MM^T = I_n\} = C_2 \times O(n).
\]
The condition on the Weyl group follows as $C_2 \ltimes O(n)/C_2 \cong O(n)$. 
\end{proof}

Similarly, we obtain the following adjunction between the intermediate categories.  Denote by $\iota: O(n) \hookrightarrow C_2 \ltimes U(n)$ the canonical subgroup inclusion induced by the inclusion $O(n) \hookrightarrow U(n)$. 

\begin{lem}
Let $n$ be a non-negative integer. There is a Quillen adjunction
\[
\adjunction{{\rm{I}}_{C_2 \times O(n)}^{C_2 \ltimes U(n)}\circ L_{\psi_n}}{\Fun_{O(n)}(\es{J}_n^\mbf{O}, O(n)\T)}{\Fun_{C_2 \ltimes U(n)}(\es{J}_n^\mbf{R}, (C_2\ltimes U(n))\T)}{\psi_n(-)^{C_2}}
\]
where ${\rm{I}}_{C_2 \times O(n)}^{C_2 \ltimes U(n)}$ is short-hand notation for the functor 
\[
(C_2 \ltimes U(n))_+\wedge_{C_2 \times O(n)}(-) : \Fun_{O(n)}(\iota^\ast\es{J}_n^\mbf{R}, O(n)\T) \to \Fun_{C_2 \ltimes U(n)}(\es{J}_n^\mbf{R}, (C_2 \ltimes U(n))\T).
\]
\end{lem}
\begin{proof}
The adjunction is the composite of the adjunction
\begin{equation}\label{eqn: adjoint 1}
\adjunction{L_{\psi_n}}{\Fun_{O(n)}(\es{J}_n^\mbf{O}, O(n)\T)}{\Fun_{O(n)}(\iota^\ast\es{J}_n^\mbf{R}, O(n)\T)}{\psi_n}
\end{equation}
with the adjunction
\begin{equation}\label{eqn: adjoint 2}
\adjunction{{\rm{I}}^{C_2 \ltimes U(n)}_{C_2 \times O(n)}}{\Fun_{O(n)}(\iota^\ast\es{J}_n^\mbf{R}, O(n)\T)}{\Fun_{C_2 \ltimes U(n)}(\es{J}_n^\mbf{R}, (C_2 \ltimes U(n))\T)}{(-)^{C_2}}
\end{equation}
where $\Fun_{O(n)}(\iota^\ast\es{J}_n^\mbf{R}, O(n)\T)$ is the category of $O(n)$-equivariant functors from $\es{J}_n^\mbf{R}$ to $O(n)\T$, with $O(n)$ acting on $\es{J}_n^\mbf{R}$ via restriction along the subgroup inclusion
\[
\iota: O(n)\hookrightarrow U(n) \hookrightarrow C_2\ltimes U(n).
\]
It suffices to demonstrate that each term in the composite is a Quillen adjunction, where the category $\Fun_{O(n)}(\iota^\ast\es{J}_n^\mbf{R}, O(n)\T)$ may be equipped with an $n$-stable model structure analogous to the $n$-stable model structures in any version of Weiss calculus. The right adjoint of the adjoint pair
\[
\adjunction{L_\psi}{\Fun_{O(n)}(\es{J}_n^\mbf{O}, O(n)\T)}{\Fun_{O(n)}(i^\ast\es{J}_n^\mbf{R}, O(n)\T)}{\psi_n}
\]
is given by $\psi_n(X)(V) = \Map_\ast(S^{i \R^n\otimes V}, X(V_\C))$, while the left adjoint is given by a suitable coend formula analogous to the left adjoint $L_{\psi}$ of the functor $\psi$ from calculus with Reality, see e.g.,~\cite[Proposition 5.5]{TaggartReality}.

The adjoint pair~\eqref{eqn: adjoint 2} is clearly a Quillen adjunction, hence it is left to show that the adjoint pair~\eqref{eqn: adjoint 1} is also. It is clear by definition that $\psi_n$ preserves levelwise (acyclic) fibrations, hence preserves acyclic fibrations of the $n$-stable model structures. It suffices to show that $\psi_n$ preserves fibrant objects as then it will preserve fibrations between fibrant objects. To see that $\psi_n$ preserves $n\Omega$-spectra, let $X \in \Fun_{O(n)}(\iota^\ast\es{J}_n^\mbf{R}, O(n)\T)$ be an $n\Omega$-spectrum, i.e., for every $V_\C, W_\C \in \es{J}_n^\mbf{R}$ the adjoint structure maps
\[
X(V_\C) \longrightarrow \Omega^{\C^n \otimes W_\C}X(V_\C \oplus W_\C)
\]
are weak equivalences of based spaces. We want to show that the adjoint structure maps 
\[
(\psi_n X)(V) \longrightarrow \Omega^{\R^n \otimes W}(\psi_n X)(V \oplus W)
\]
are weak homotopy equivalences for all $V,W \in \es{J}_n^\mbf{O}$. This follows from writing these structure maps as
\[
\Omega^{i(\R^n \otimes V)}X(V_\C) \longrightarrow \Omega^{\R^n \otimes W}\Omega^{i(\R^n \otimes (V \oplus W))}X(V_\C \oplus W_\C)
\]
and using that $X$ is an $n\Omega$-spectrum to write the left-hand side as 
\[
\Omega^{i(\R^n \otimes V)}\Omega^{\C^n \otimes W_\C}X(V_\C \oplus W_\C)
\]
and comparing the results by identifying $\C^n$ with $\R^n \otimes i\R^n$.
\end{proof}

\subsection{The commutativity of Figure~\ref{fig: model categories}}

We begin with the commutativity of the top pentagon ($\circled{5}$) of Figure~\ref{fig: model categories}. Recall that given a topological model category $\es{M}$, the category $\es{M}^\sf{fc}$ of fibrant-cofibrant objects defines a topological category, the nerve of which defines a quasi-category. We refer to the quasi-category $N(\es{M}^\sf{fc})$ as the \emph{underlying} $\infty$-category of $\es{M}$, and denote it by $\es{M}_\infty$. Given a Quillen adjunction (\emph{resp.} Quillen equivalence) of model categories $\adjunction{F}{\es{M}}{\es{N}}{G}$, there is a well-defined adjunction (\emph{resp.} equivalence) of underlying $\infty$-categories, $\adjunction{\mathds{L}F}{\es{M}_\infty}{\es{N}_\infty}{\mathds{R}G}$.

\begin{prop}
Let $n$ be a non-negative integer. The diagram
\[\begin{tikzcd}
	{\s^\mbf{O}[C_2 \ltimes U(n)]_\infty} && {\s^\mbf{O}[O(n)]_\infty} \\
	{\s^\mbf{R}[U(n)]_\infty} \\
	{\Fun_{C_2 \ltimes U(n)}(\es{J}_n^\mbf{R}, (C_2 \ltimes U(n))\T)_\infty} && {\Fun_{O(n)}(\es{J}_n^\mbf{O}, O(n)\T)_\infty}
	\arrow["{\mathds{R}\psi}"', shift right=1, from=2-1, to=1-1]
	\arrow["{\mathds{L}L_{\psi}}"', shift right=1, from=1-1, to=2-1]
	\arrow["{\mathds{R}(\alpha_n^\mbf{R})^\ast}", shift left=1, from=2-1, to=3-1]
	\arrow["{\mathds{L}(\alpha_n^\mbf{R})_!}", shift left=1, from=3-1, to=2-1]
	\arrow["{\mathds{R}(\alpha_n^\mbf{O})^\ast}", shift left=1, from=1-3, to=3-3]
	\arrow["{\mathds{L}(\alpha_n^\mbf{O})_!}", shift left=1, from=3-3, to=1-3]
	\arrow["{\mathds{R}(-)^{C_2}}"', shift right=1, from=1-1, to=1-3]
	\arrow["{(C_2 \ltimes U(n))_+ \wedge^\mathds{L}_{C_2 \times O(n)}(-)}"', shift right=1, from=1-3, to=1-1]
	\arrow["{\mathds{R}(\psi_n(-)^{C_2})}"', shift right=1, from=3-1, to=3-3]
	\arrow["{(C_2 \ltimes U(n))_+ \wedge^\mathds{L}_{C_2 \times O(n)}(L_{\psi_n}(-))}"', shift right=1, from=3-3, to=3-1]
\end{tikzcd}\]
is a commutative diagram of $\infty$-categories up to natural equivalence.
\end{prop}
\begin{proof}
Consider the corresponding diagram of model categories, and let $X$ be an object of $\s^\mbf{R}[U(n)]$. For $V$ in $\es{J}_n^\mbf{O}$, we have that
\[
((\alpha_n^\mbf{O})^\ast \circ (-)^{C_2} \circ \psi)X(V) = \Map_\ast(S^{i(\R^n \otimes V)}, X(\C^n \otimes V_\C))^{C_2}
\]
and 
\[
(\psi_n \circ (-)^{C_2} \circ (\alpha_n^\mbf{R})^\ast)X(V) = \Map_\ast(S^{i(\R^n \otimes V)}, X(\C^n \otimes V_\C)^{C_2}).
\]
If $X(\C^n \otimes V_\C)$ is fibrant in the cofree model structure, then the above two mapping spaces are naturally isomorphic, hence it suffices to exhibit that the derived image of $X$ in $\Fun(\es{J}_n^\mbf{O}, O(n)\T)_\infty$ is cofree. This follows readily from the fact that the image of $X$ in $\Fun(\es{J}_n^\mbf{O}, O(n)\T)_\infty$ is the image of $X$ under a series of right derived functors.
\end{proof}

We conclude by demonstrating that the square ($\circled{4}$) commutes on the level of underlying $\infty$-categories.

\begin{prop}
Let $n$ be a non-negative integer. The diagram
\[\begin{tikzcd}
	{\Fun_{C_2 \ltimes U(n)}(\es{J}_n^\mbf{R}, (C_2 \ltimes U(n))\T)_\infty} & {} && {\Fun_{O(n)}(\es{J}_n^\mbf{O}, O(n)\T)_\infty} \\
	{\homog{n}(\es{J}_0^\mbf{R}, C_2\T)_\infty} & {} && {\homog{n}(\es{J}_0^\mbf{O},\T)_\infty}
	\arrow["{\mathds{R}\ind_0^n\varepsilon^\ast}"', shift right=1, from=2-4, to=1-4]
	\arrow["{\mathds{L}\res_0^n/O(n)}"', shift right=1, from=1-4, to=2-4]
	\arrow["{\mathds{L}\res_0^n/U(n)}"', shift right=1, from=1-1, to=2-1]
	\arrow["{\mathds{R}\ind_0^n\varepsilon^\ast}"', shift right=1, from=2-1, to=1-1]
	\arrow["{(C_2 \ltimes U(n))_+ \wedge^\mathds{L}_{C_2 \times O(n)}(L_{\psi_n}(-))}"', shift right=1, from=1-4, to=1-1]
	\arrow["{\mathds{R}(\psi_n (-)^{C_2})}"', shift right=1, from=1-1, to=1-4]
	\arrow["{\mathds{L}(i^\ast c_!)}"', shift right=1, from=2-4, to=2-1]
	\arrow["{\mathds{R}(c^\ast(-)^{C_2})}"', shift right=1, from=2-1, to=2-4]
\end{tikzcd}\]
is a commutative diagram of $\infty$-categories, up to natural equivalence.
\end{prop}
\begin{proof}
It suffices to show that that the diagram
\[\begin{tikzcd}
	{\s^\mbf{O}[C_2 \ltimes U(n)]_\infty} && {} & {\s^\mbf{O}[O(n)]_\infty} \\
	{\s^\mbf{R}[U(n)]_\infty} \\
	{\Fun_{C_2 \ltimes U(n)}(\es{J}_n^\mbf{R}, (C_2 \ltimes U(n))\T)_\infty} & {} && {\Fun_{O(n)}(\es{J}_n^\mbf{O}, O(n)\T)_\infty} \\
	{\homog{n}(\es{J}_0^\mbf{R}, C_2\T)_\infty} && {} & {\homog{n}(\es{J}_0^\mbf{O},\T)_\infty} \\
	\arrow["{\mathds{R}\psi}"', shift right=1, from=2-1, to=1-1]
	\arrow["{\mathds{L}L_{\psi}}"', shift right=1, from=1-1, to=2-1]
	\arrow["{\mathds{R}\ind_0^n\varepsilon^\ast}"', shift right=1, from=4-1, to=3-1]
	\arrow["{\mathds{L}\res_0^n/U(n)}"', shift right=1, from=3-1, to=4-1]
	\arrow["{\mathds{R}(\alpha_n^\mbf{R})^\ast}", shift left=1, from=2-1, to=3-1]
	\arrow["{\mathds{L}(\alpha_n^\mbf{R})_!}", shift left=1, from=3-1, to=2-1]
	\arrow["{\mathds{R}(c^\ast(-)^{C_2})}"', shift right=1, from=4-1, to=4-4]
	\arrow["{\mathds{L}(i^\ast c_!)}"', shift right=1, from=4-4, to=4-1]
	\arrow["{\mathds{R}(-)^{C_2}}"', shift right=1, from=1-1, to=1-4]
	\arrow["{(C_2 \ltimes U(n))_+ \wedge^\mathds{L}_{C_2 \times O(n)}(-)}"', shift right=1, from=1-4, to=1-1]
	\arrow["{\mathds{L}(\alpha_n^\mbf{O})_!}", shift left=1, from=3-4, to=1-4]
	\arrow["{\mathds{R}(\alpha_n^\mbf{O})^\ast}", shift left=1, from=1-4, to=3-4]
	\arrow["{\mathds{L}\res_0^n/O(n)}"', shift right=1, from=3-4, to=4-4]
	\arrow["{\mathds{R}\ind_0^n\varepsilon^\ast}"', shift right=1, from=4-4, to=3-4]
\end{tikzcd}\]
commutes, by which we mean that there is a natural equivalence of composites
\[
\mathds{L}(\alpha_n)_! \circ \mathds{R}\ind_0^n\varepsilon^\ast \circ \mathds{R}(c^\ast(-)^{C_2}) \simeq \mathds{R}(-)^{C_2} \circ \mathds{R}\psi \circ \mathds{L}(\xi_n)_! \circ \mathds{R}(\ind_0^n\varepsilon^\ast).
\]
By an adjoint argument using that the vertical adjunctions are equivalences of $\infty$-categories, it suffices to exhibit a natural equivalence
\[
\mathds{R}(c^\ast(-)^{C_2}) \circ \mathds{L}\res_0^n/U(n) \circ \mathds{R}(\xi_n)^\ast \circ \mathds{L} L_\psi \simeq \mathds{L}\res_0^n/O(n) \circ \mathds{R}(\alpha_n)^\ast \circ \mathds{R}(-)^{C_2}. 
\]
Let $\Theta$ be a bifibrant object in the model category of  spectra with an action of $C_2 \ltimes U(n)$. By the classification theorem for calculus with Reality, see~\cite[Theorem 7.1]{TaggartReality}, the left-hand side applied to $\Theta$ is naturally levelwise weakly equivalent to the functor
\[
V \longmapsto (\Omega^\infty [(S^{\C^n \otimes_\C V_\C} \wedge \Theta)_{hU(n)}])^{C_2},
\]
whereas by the classification theorem for orthogonal calculus, see~\cite[Theorem 7.3]{We95} and \cite[Theorem 10.3]{BO13}, the right-hand side is naturally levelwise weakly equivalent to the functors
\[
V \longmapsto \Omega^\infty [(S^{\R^n \otimes_\R V} \wedge \Theta^{C_2})_{hO(n)}],
\]
hence it suffices to exhibit a natural isomorphism between these functors. First note that $C_2$-fixed points commutes with $\Omega^\infty$, hence there is a natural equivalence
\[
(\Omega^\infty [(S^{\C^n \otimes_\C V_\C} \wedge \Theta)_{hU(n)}])^{C_2} \simeq \Omega^\infty [(S^{\C^n \otimes_\C V_\C} \wedge \Theta)_{hU(n)}]^{C_2},
\]
which by definition of the homotopy orbits we may rewrite as
\[
(\Omega^\infty [(S^{\C^n \otimes_\C V_\C} \wedge \Theta) \wedge_{U(n)} EU(n)_+])^{C_2} \simeq \Omega^\infty [(S^{\C^n \otimes_\C V_\C} \wedge \Theta) \wedge_{U(n)} EU(n)_+]^{C_2}.
\]
Since $C_2$ acts diagonally on the right-hand side, there is a natural isomorphism
\[
[(S^{\C^n \otimes_\C V_\C} \wedge \Theta) \wedge_{U(n)} EU(n)_+]^{C_2} \simeq ((S^{\C^n \otimes_\C V_\C})^{C_2} \wedge \Theta^{C_2}) \wedge_{U(n)^{C_2}} (EU(n)_+)^{C_2},
\]
which by calculation is naturally isomorphic to 
\[
(S^{\R^n \otimes_\R V} \wedge \Theta^{C_2}) \wedge_{O(n)} EO(n)_+ = (S^{\R^n \otimes_\R V} \wedge \Theta^{C_2})_{hO(n)},
\]
hence the result follows.
\end{proof}

\bibliographystyle{alpha}
\bibliography{references.bib}
\end{document}